\documentclass[oneside,reqno]{amsart}
\usepackage{amsthm}
\usepackage{amssymb}
\usepackage{amsmath,amssymb}
\usepackage{amsmath} 
\usepackage{color}
\theoremstyle{plain}
\newtheorem{theorem}{Theorem}[section]
\newtheorem{corollary}[theorem]{Corollary}
\newtheorem{lemma}[theorem]{Lemma}
\newtheorem{proposition}[theorem]{Proposition}
\theoremstyle{definition}
\newtheorem{definition}[theorem]{Definition}
\newtheorem{remark}[theorem]{Remark}
\newtheorem{example}[theorem]{Example}
\newtheorem{assumption}[theorem]{Assumption}
\numberwithin{equation}{section}

\newcommand{\fHS}{\mathcal{H}_s^\infty}
\newcommand{\DfHS}{\mathcal{H}_s^{-\infty}}
\newcommand{\coLp}{\mathrm{ Co}_{\rho_s}^u L^p_{\alpha+n+1-s p/2}(G)}
\newcommand{\SU}{\mathrm{SU}}

\newcommand{\B}{\mathbb{B}^n}
\newcommand{\supp}{\mathrm{supp}}
\newcommand{\Span}{\mathrm{span}}
\newcommand{\osc}{\mathrm{osc}}
\newcommand{\fH}{\mathcal{H}}
\newcommand{\fS}{\mathcal{S}}
\newcommand{\T}{\mathbb{T}}
\newcommand{\C}{\mathbb{C}}
\newcommand{\R}{\mathbb{R}}  
\newcommand{\hH}{\widehat{H}}

\newcommand{\fg}{\mathfrak{g}}

\newcommand{\coBP}{\mathrm{ Co}_{\rho}^{u}B}

\newcommand{\coBh}{\mathrm{ Co}_\mathcal{\rho_\sigma}^{u} \widehat{B}}
\newcommand{\ft}{\widetilde{f}}

\newcommand{\Bh}{\widehat{B}}
\newcommand{\cH}{\mathcal{H}}
\newcommand{\cS}{\mathcal{S}}
\newcommand{\Id}{ \mathrm{d}}
\newcommand{\id}{\mathrm{id}}
\newcommand{\Iff}{ \mbox{if and only if }}
\newcommand{\llangle}{\langle\!\langle}
\newcommand{\rrangle}{\rangle\!\rangle}
\newcommand{\GL}{\mathrm{GL}}
\newcommand{\Co}{\mathrm{Co}}
\newcommand{\ip}[2]{ ( #1,#2 )}

\title[Coorbits for projective representations]{Coorbits for projective representations with an application to Bergman spaces }

\author{Jens Gerlach Christensen} 
\address{ Department of Mathematics, Colgate University, 13 Oak Drive, Hamilton NY 13346} 
\email{jchristensen@colgate.edu}
\urladdr{http://www.math.colgate.edu/~jchristensen}

\author{Amer H. Darweesh} 
\address{Department of Mathematics and Statistics,
Jordan University of Science and Technology,
Irbid 22110, B.O. Box (3030) Jordan.
  } 
\email{ahdarweesh@just.edu.jo} 
 
  \author{Gestur \'Olafsson} 
\address{ Department of Mathematics, Louisiana State University, Baton Rouge, LA 70803} 
\email{olafsson@math.lsu.edu}
\urladdr{http://www.math.lsu.edu/~olafsson}

\thanks{The research was partially supported by NSF grant   DMS 1321794. The first and last named authors would
also like to thank AMS for it's support during the MRC program 
\textit{Lie Group Representations, Discretization, and Gelfand Pairs}
June 5--June 11, 2016 }

\begin{document}
\begin{abstract}
Representation theory of locally compact topological groups is
a powerful tool to analyze Banach spaces of
functions and distributions. It provides a unified framework for
constructing function spaces and to study several generalizations
of the wavelet transform.
Recently representation theory has been used to provide
atomic decompositions for a large collection of classical
Banach spaces. But in some natural situations, including Bergman spaces
on bounded domains, representations are too restrictive. The proper
tools are projective representations.
 In this paper we extend known techniques from representation theory
to also include projective representations. This leads naturally to
twisted convolution on groups avoiding the usual central extension of
the group.
As our main application we obtain
atomic decompositions of Bergman spaces on
the unit ball through the holomorphic discrete series
for the group of isometries of the ball.
\end{abstract}

\maketitle

\section{Introduction}
\noindent
With the rise of continuous wavelet theory it was discovered 
that representation theory could be used to obtain atomic
decompositions for some classical Banach spaces. 
This area of harmonic analysis is called coorbit theory
and it was initiated by Feichtinger and Gr\"ochenig
\cite{Feichtinger1988,Feichtinger1989,Feichtinger1989a,Grochenig1991}.
Several interesting generalizations were later presented in
\cite{Christensen1996,Rauhut2005,Fornasier2005,Rauhut2007a,Dahlke2004a,Dahlke2004,Dahlke2007,Dahlke2008}. All these examples use irreducible 
integrable representations
in order to construct atomic decompositions. This allows one to choose
atoms in an appropriate minimal Banach space. 
As has been remarked recently, assuming integrability and irreducibility is not needed, and in fact often the 
restriction of irreducibility and  integrability as well as the criteria for
selecting atoms turns out to be too restrictive. 
Therefore, the first and last author suggested the use of Fr\'echet spaces \cite{Christensen2011,Christensen2012} in coorbit 
theory, and recently the idea has been used in several cases
\cite{Christensen2012a,Christensen2013,Christensen2016,Dahlke2016}.

There are many situations in which projective representations arise
more naturally than representations. 
Typical examples are the modulation spaces and the
short time Fourier transform which stems from modulation and translation,
as well as the holomorphic discrete series representations 
on Bergman spaces on bounded symmetric domains.
The idea of using projective representations in coorbit theory
has earlier been explored in \cite{Christensen1996} under the assumptions
of irreducibility, integrability, and  continuity of the
multiplier. As mentioned, our aim is to present and apply
a coorbit theory without these restrictions.
The first two restrictions were removed in the thesis of
the second author \cite{Darweesh2015}.
The continuity assumption used in both
\cite{Christensen1996} and \cite{Darweesh2015}
means that those approaches apply 
only to some special cases like Abelian groups
or simply connected groups.
For many simply connected groups this creates new 
obstacles due to them having infinite center.
As a consequence, the papers \cite{Christensen1996} and \cite{Darweesh2015} 
cannot be used to describe Bergman spaces on the unit ball in $\mathbb{C}^n$
as coorbits for the group $\mathrm{SU}(n,1)$, or more generally, 
Bergman spaces on bounded domains in $\mathbb{C}^n$.

This is also the reason that we had to work with
finite covering groups of $\mathrm{SU}(n,1)$ and 
require rationality of the representation parameter in \cite{Christensen2016}.
Bargmann and Mackey have shown that 
for general locally compact groups the multiplier
can be chosen continuous in a neighborhood of the identity,
and for Lie groups the multiplier can even be chosen smooth in
such a neighborhood.
In this paper we show that these facts are sufficient for
obtaining a working coorbit theory for projective representations.
Furthermore, we demonstrate the benefits of the theory by applying it to the
case of Bergman spaces on the unit ball.
This finishes the work initiated in \cite{Christensen2016},
since it allows us to remove the rationality restriction
on the representation parameter,
and thereby we can provide atomic decompositions for 
the entire scale of Bergman spaces.
 The approach extends to Bergman spaces on general bounded symmetric
 domains, which will appear in a forthcoming paper by the first and last authors.

We would like to also point out another 
difference between the present paper and \cite{Christensen1996}
and \cite{Darweesh2015}. In those papers the atomic decompositions
are obtained by applying results in  \cite{Feichtinger1989}
and \cite{Christensen2012}, respectively, to the Mackey obstruction group
(which will be introduced in the next section).
In this paper we avoid this difficulty and work directly on the group. 
Here the reader should keep in mind the modulation spaces of Feichtinger 
\cite{Feichtinger1983,Feichtinger2006a} 
(see also the book \cite{Grochenig2001}). 
Those spaces arise
from translation and modulation on functions or distributions on $\R^{n}$. But those two actions do
not commute and lead in a natural way to the Schr\"odingar representation of the reduced Heisenberg group. But the theory
is usually carried out without any mention of the compact center of the reduced Heiseinberg group, and only the action of
$\R^{2n}$, the corresponding cocylce $e^{-ix\cdot y}$, and the twisted convolution are used.
For modulation spaces on general Abelian groups see 
\cite{Feichtinger2003a,Grochenig1998}.

\section{Projective Representations}\noindent
In this section we review known results on measurable and locally continuous
projective representations, and, following Mackey, 
we will construct a representation of an extension of a locally compact
second countable group from a given projective representation. We use \cite{Varadarajan1985} as
a standard reference even if the results are mostly due to Mackey and Bargmann. 

We assume that $G$ is a locally compact second countable group equipped with a fixed left
invariant Haar measure which we denote by $dx$.
We denote by $\mathbb{T}:=\{t \in \C \mid |t|=1\}$ the one dimensional torus
with normalized Haar measure $dt$.

\begin{definition}
  Let $ \mathcal{S}$ be a locally convex Hausdorff topological vector space over $\C$,
  and denote by $\mathcal{S}^*$ its conjugate dual. A projective representation of 
  $G$ is a mapping $\rho : G \to \GL(\mathcal{S})$, the space of
continuous linear maps $\cS\to \cS$ with a continuous inverse, that satisfies the following three
  conditions:
  \begin{enumerate}
  \item $\rho(1)= \id.$
  \item There is a Borel function (called a \emph{multiplier} or a \emph{cocycle})
    $ \sigma:G \times G \to \mathbb{T}$, which satisfies the condition 
    $$\rho(xy)=\sigma (x,y) \rho(x)\rho(y).$$ 
  \item For every $v\in \mathcal{S}$ and every $\lambda\in S^*$
    the mapping 
    \[x \mapsto \langle\lambda, \rho(x)v\rangle :=\lambda (\rho (x)v)\]
is a Borel function.
\end{enumerate}

If there is a neighbourhood around $e$ on which 
the function $x\mapsto \langle\lambda, \rho(x)v\rangle$ is continuous
for all $\lambda\in\fS^*$, we say that $\rho$ is locally weakly continuous.
A projective representation with multiplier $\sigma$ is said to be a {\it $\sigma$-representation}.
We call the multiplier $\sigma$ locally
continuous if it is continuous on a neighbourhood of $e\times e$.
\end{definition}
Notice, that for a $\sigma$-representation $\rho$ 
and $x\in G$ we have
\begin{equation}\label{eq:rhoInv}
\rho(x^{-1}) = \overline{\sigma (x,x^{-1})} \rho(x)^{-1}= \overline{\sigma(x^{-1},x)}\rho(x)^{-1}\, .
\end{equation}

Let $x,y,z\in G$. The following are straightforward consequences about the cocycle $\sigma$: 
\begin{enumerate}
  \item $ \sigma(x,1)=\sigma (1,x) =1$, 
  \item $ \sigma(x,y) ^{-1} = \overline{\sigma(x,y)}$,
    \item $\sigma(x,x^{-1}) = \sigma(x^{-1},x)$,
  \item $\sigma(xy,z) \sigma(x,y)=\sigma(x,yz)\sigma(y,z)$.
\end{enumerate}

Following Bargmann, we say that two cocycles $\sigma$ and $\tau$
are \emph{similar} if there exists a Borel function $a:G\to\T$
such that
\begin{equation}\label{eq:a}
\tau(x,y) = \frac{a(xy)}{a(x)a(y)}\sigma(x,y).
\end{equation}
We  note that
if $\tau$ is similar to $\sigma$ via the Borel function $a$ and $\rho$ is a  $\sigma$-representation, then
$\eta (x ):= a(x)\rho (x)$ is a $\tau$-representations.
 
\begin{theorem} Every multiplier is similar to a multiplier which is continuous on some open neighborhood
of $(e, e)\in G\times G$. If $G$ is a Lie group, then every multiplier is similar to a multiplier that is smooth
in an open neighborhood of
$(e,e)$. If $G$ is a connected and simply connected Lie group, then every multiplier is similar to a multiplier which is analytic on the entire group $G\times G$.
\end{theorem}
\begin{proof} The first statement is  \cite[Corollary 7.6]{Varadarajan1985}. The second statement is \cite[Lemma 7.20]{Varadarajan1985} and
the last segment is \cite[Corollary 7.30]{Varadarajan1985}.
\end{proof}
 
For projective representations we define irreducibility, cyclicity, admissible vectors, unitarity and square integrability 
in the same way as for representations. We summarize these notions in the following definition.
\begin{definition}
Let $(\rho,\fS)$ be a projective representation, then
\begin{enumerate}
\item A subspace $W$ of $\fS$ is $\rho$-invariant if $\rho(x)W\subseteq W$ for all $x\in G$.
\item $(\rho,\fS)$ is irreducible if the only closed $\rho$-invariant subspaces are $\{0\}$ and $\fS$ itself.
\item A vector $u \in \fS$ is a $\rho$-cyclic if $\Span\{\rho (x)u \mid x \in G\}$ is dense in $\fS$. If such a cyclic vector exists, we say 
that $(\rho,\fS)$ is a cyclic projective representation.
\item $(\rho,\fH)$ is unitary if $\fH$ is a Hilbert space, and $\rho(x)$ is a unitary operator for every $x\in G$.
\item If $\fH$ is a Hilbert space and 
$(\rho,\fH)$ is irreducible and unitary it is called square-integrable if there is a nonzero vector $u\in \fH$ such that $$\int_G |\langle u,\rho(x)u\rangle|^2\, \,d x< \infty.$$ 
In this case $u$ is called a \emph{$\rho$-admissible vector}.
\end{enumerate}
\end{definition}

 In the following lemma we define the dual 
projective representation on the conjugate dual of a Fr\'echet space.

Let $\fS$ be a Fr\'echet space 
and denote by $\fS^*$ be  the conjugate dual of $\fS$ equipped with the
weak*-topology. This implies that $\fS^*$ is a locally convex vector
space and $(\fS^*)^*=\fS$.

\begin{lemma} 
Let $(\rho,\fS)$ be a $\sigma$-representation of $G$ on a Fr\'echet space $\fS$, 
and let $\fS^*$ be  the conjugate dual of $\fS$ equipped with the weak*-topology. 
The mapping $\rho^*$, which is given by
$$\langle\rho^*(x)\lambda,v\rangle:=\langle\lambda,\rho(x)^{-1}v\rangle$$ 
for all $\lambda\in \fS^*$ and all $v\in \fS$, defines a $\sigma$-representation of $G$ on the space $\fS^*$. 
Finally $x\mapsto \langle \rho^*(x)\lambda ,u\rangle$ is continuous
around $e$ if $\sigma$ is locally continuous and $\rho$ is locally weakly continuous.
\end{lemma}
\begin{proof}
As $\fS^*$ is equipped with the weak$^*$ topology, its conjugate
dual is $\fS$. We denote this dual pairing by
$\llangle\cdot,\cdot\rrangle$.
For $v\in (\fS^*)^* = \fS$ and $\lambda \in \fS^*$
we have
$\llangle v,\lambda\rrangle = \overline{\langle \lambda,v\rangle}.$
Therefore the following calculation 
shows that $\rho^*$ has cocycle $\sigma$
\begin{align*}
\langle\rho^*(xy)\lambda,v\rangle=& \langle\lambda,\rho(xy)^{-1}v\rangle\\
=&\langle\lambda,(\sigma(x,y)\rho(x)\rho(y))^{-1}v\rangle\\
=&\langle\lambda,\overline{\sigma(x,y)} \rho(y)^{-1}\rho(x)^{-1}v\rangle\\
=&\langle\sigma(x,y)\lambda, \rho(y)^{-1}\rho(x)^{-1}v\rangle\\
=&\langle\sigma(x,y)\rho^*(x)\rho^*(y)\lambda,v\rangle 
\end{align*}
Hence, $\rho^*(xy)=\sigma(x,y)\rho^*(x)\rho^*(y)$.   
Furthermore,
the equation (\ref{eq:rhoInv}) implies that
\[ 
\llangle v,\rho^*(x)\lambda\rrangle = 
\overline{\langle \rho^* (x)\lambda , v\rangle}
=\overline{\langle \lambda , \rho(x)^{-1} u\rangle}
=
\sigma (x^{-1},x) \overline{\langle \lambda ,  \rho (x^{-1}) u\rangle}\, .
\]
Since $x\mapsto x^{-1}$ is continuous and $\sigma$ and $\rho$ are Borel 
it follows that $ x\mapsto \llangle v,\rho^* (x)\lambda\rrangle$
is Borel.
Moreover, the mapping is locally continuous 
if $\sigma $ locally continuous and $\rho$ is locally weakly continous.
\end{proof}

For any projective representation $\rho$ of $G$ we
can construct an actual representation of a new group related to $G$
which is called the Mackey obstruction group of $G$ (see p. 269f in \cite{Mackey1958}).
We refer to \cite[Chap. VII]{Varadarajan1985} for a detailed discussion.

We first gather some facts about the Mackey group.
Let $\sigma$ be a cocycle on $G\times G$.
As a set, the Mackey group that corresponds to $G$ is the
group $G_\sigma:=G\times\T$, with multiplication given
by 
$$(x,t)(y,s)=(xy,\overline{\sigma(x,y)}\,\, ts)\, .$$ 
The inverse of $(x,t)\in G_\sigma$ is given by 
$$
(x,t)^{-1} = (x^{-1},\sigma(x,x^{-1})\overline{t}) =
(x^{-1},\sigma(x^{-1},x)\overline{t}).
$$
Note that this is a {\it central} extension of $\T$ by $G$ as all the elements $(e, s)$, $s\in\T$, are central
in $G_\sigma$. 
The product of the Borel algebras of $G$ and $\T$ defines a $\sigma$-algebra on $G_\sigma$ and the
product measure $dxdt$ is left invariant.
\begin{theorem}\label{Weyltopology} Let the notation be as above. Then the following holds:
\begin{enumerate}
\item \label{WT-1}
There exists a unique topology on $G_\sigma$, called the \textit{Weyl topology}, that generates the product  $\sigma$-algebra and 
at the same time makes 
$G_\sigma$ into a locally compact Hausdorff topological group.
\item \label{WT-2}
If the multiplier $\sigma$ is continuous around $(e,e)$, then there exists a neighborhood $U$ of $e$ in $G$ such that the Weyl topology on $U\times \T$ corresponds
to the product topology on $U\times \T$.
\item \label{WT-3} 
Two extensions $G_{\sigma}$ and $G_{\tau}$ are isomorphic if and only if
the multipliers $\sigma$ and $\tau$ are similar. If the similarity is given by the function $a$ as in (\ref{eq:a}) then the isomorphism $G_\sigma \to G_\tau$ is given be
$(x,t)\mapsto (x,\overline{a(x)}t)$.
\item \label{WT-4} 
  If $G$ is a connected Lie group, then there exists a unique
analytic structure on $G_\sigma$ compatible with the Weyl topology. The maps
$t\to (e,t)$ and $(x,t)\to x$ are analytic.
\item \label{WT-5}
There exists a smooth, respectively analytic, multiplier $\tau$ similar to $\sigma$ if and only if there exists a smooth, respectively analytic,
map  $\kappa : G\to G_\sigma$ such that $p_1(\kappa (x))=x$ where $p_1(x,t)=x$. 
\end{enumerate}
\end{theorem}

\begin{proof} 
(\ref{WT-1}) and (\ref{WT-3}) are \cite[Theorem 7.8]{Varadarajan1985}. 
(\ref{WT-2}) is \cite[Corollary 7.10]{Varadarajan1985}. 
(\ref{WT-4}) is \cite[Theorem 7.21]{Varadarajan1985}.
Finally, (\ref{WT-5}) is \cite[Corollary 7.23]{Varadarajan1985}.
\end{proof}  
Next we discuss the construction of a representation of $G_\sigma$
from a $\sigma$-representation $\rho$. Define
\[\rho_\sigma (x,t):= t\rho (x)\, .\]
Then a simple calculation shows that $\rho_\sigma$ is a homomorphism. 
If $\tau$ is similar to $\sigma$,  $\eta (x)=a(x)\rho (x)$ is the corresponding canonical $\tau$-representation and
$\varphi (x,t)=(x,\overline{a(x)}t)$ is the natural isomorphism $G_\sigma \simeq G_\tau$, then
\[\eta_\tau (\varphi (x,t))=\overline{a(x)}t\eta (x) = \overline{a(x)}a(x)t\rho (x)=\rho_\sigma (x,t)\]
as $a(x)\in\T$. Hence $\eta_\tau\circ \varphi =\rho_\sigma$.

The following is well known
for representations on a separable Hilbert space, see \cite[Theorem 7.16]{Varadarajan1985}.

\begin{theorem} \label{thm:weakcontinuity}
  Assume that the multiplier $\sigma$ is continuous in a neighborhood of
  $(e, e)$,  and that $G$ is connected.
  Then 
\begin{equation}\label{eq:rhoSigma}
G_\sigma \to \C\, ,\quad (x,t)\mapsto \langle \lambda ,\rho_\sigma (x,t) u\rangle=\overline t \langle \lambda ,\rho  (x)u\rangle
\end{equation}
is a Borel function
for all $\lambda\in\cS^*$ and $u\in\cS$. 
Furthermore the following holds:
\begin{enumerate}
\item The map in  (\ref{eq:rhoSigma}) is Borel.
\item The map in  (\ref{eq:rhoSigma}) is 
continuous if and only if $G\to \C$, $x\mapsto \langle \lambda, \rho (x)u\rangle$ is continuous in an open neighborhood around $e$.
\item Assume that $G$ is a Lie group. The map in  (\ref{eq:rhoSigma})  is smooth if and only if $G\to \C$, $x\mapsto \langle \lambda, \rho (x)u\rangle$ is smooth in an open neighborhood around $e$.
\end{enumerate}
\end{theorem}
\begin{proof}
The first statement follows from the fact that it is the product
    of the two Borel maps
    $$
    (x,t)\mapsto \overline{t}\qquad\text{and}\qquad
    (x,t)\mapsto \langle \lambda,\rho(x)u\rangle,
    $$
    and that the Borel $\sigma$-algebra on $G_\sigma$ is the
    product of the Borel algebras of $G$ and $\mathbb{T}$.

  If $\sigma$ is continuous around $(e, e)$
  there exists, according to Theorem \ref{Weyltopology}, part (2),
  an $e$-neighborhood $U$ in $G$ such that
  the Weyl topology on $U\times \T$ agrees with the product topology.
  Thus the map in (\ref{eq:rhoSigma}) is continuous on $U\times \T$.
  Let $z=(y , s )\in G_\sigma$.
  Then $U\times \T(y, 1)$ is a neighborhood
  around $z$ and for $(x , t)\in U\times \T$ we have 
  \[(x, t)\mapsto \lambda( \rho_\sigma((x, t)  (y, 1) )u )=  \lambda( \rho_\sigma((x, t) ) \rho_\sigma((y, 1) )u )\]
  which is continous in $(x ,t )$.
  Hence the map in (\ref{eq:rhoSigma}) is continuous.  
Let us check the opposite direction and assume that
the mapping (\ref{eq:rhoSigma}) is continuous. 
Restrict the mapping to a neighbourhood $U\times \T$ 
in the Weyl topology for which $U$ is open in $G$.
The mapping from $U$ to $U\times\T$ given by
$x\to (x,1)$ is then continuous, and therefore
$x\mapsto \langle \lambda,\rho_\sigma(x,1)u\rangle 
= \langle \lambda,\rho(x)u\rangle$
is continuous on the neighbourhood $U$.

Smoothness is verified in the same manner.\end{proof} 
 
\begin{assumption}
\label{assumption-cont}
{}From now on we will assume that multipliers are continuous in a neighborhood of $(e,e)$.
Moreover, we will assume that for every $\lambda\in \cS^*$ and every $u\in \cS$, 
there is a neighbourhood $U$ of $e$ on which
the mapping $x\mapsto \langle \lambda, \rho(x)u\rangle$ is continuous. Finally the group $G$ is assumed to be connected.
\end{assumption}

The following examples show that there are plenty of examples
for which this assumption is satisfied.

\begin{example}
  \label{ex:hilbertcontinuity}
  If $\fH$ is a Hilbert space and $\rho$ is a unitary projective
  representation, then Corollary 7.10 in \cite{Varadarajan1985} 
  ensures that there is a neighbourhood $U$ on which 
  $x\mapsto \langle \lambda, \rho(x)u\rangle $ is continuous.
  
  We also note that the representation $\rho_\sigma$ is unitary if and only if the $\sigma$-projective representation
  $\rho$ is unitary. Finally, the $\sigma$-representation $\rho$ is square-integrable if and only if $\rho_\sigma$-is square
integrable. This last statement follows from the fact that $\ip{u}{\rho_{\sigma}(x,t)u}=\overline{t}\ip{u}{\rho (x)u}$ and that $\T$ is compact with measure
$1$. Thus $\int_{G_\sigma}|\ip{u}{\rho_{\sigma}(x,t)u}|^2\, dxdt =\int_G |\ip{u}{\rho (x)u}|^2\, dx$
\end{example}

We now provide a version of the Duflo-Moore theorem for square integrable
projective representations. This result can be found in \cite{Aniello2006}.
\begin{theorem}\label{thm:duflomoore}
  Let $(\rho, \fH)$ be a square-integrable projective representation
  of $G$.
  \begin{enumerate}
  \item There exists a positive self adjoint operator $A_\rho$ which
    is defined on a dense subset $D$ of $\fH$, such that $u\in \fH$ is
    $\rho$-admissible \Iff \, $u\in D$. Moreover, the orthogonality
    relation
$$\int_G (v_1,\rho(x)u_1)\,(\rho(x)u_2,v_2)\, \,d x=(A_\rho{u_2},A_\rho{u_1})\,(v_1,v_2)$$
holds for all $u_1,u_2 \in D$ and $v_1,v_2 \in \fH$.
\item In addition, if $G$ is a unimodular, then $D=\fH$ and
  $A_\rho=c_\rho Id_{\fH}$. Thus, all vectors of $\fH$ are
  $\rho$-admissible and
$$\int_G (v_1,\rho(x)u_1)\,(\rho(x)u_2,v_2)\, \,d x=c_\rho^2({u_2},{u_1})\,(v_1,v_2)$$ 
for all $u_1,u_2 ,v_1,v_2 \in \fH $. The constant $1/c_\rho^2$ is called
the formal dimension of $\rho$.
\end{enumerate}
\end{theorem}

\begin{example}
  We will now show that for a specific Gelfand triple
  $(\cS,\cH,\cS^*)$ from the original coorbit theory 
  \cite{Feichtinger1988,Feichtinger1989,Feichtinger1989a,%
    Christensen1996}
  the weak continuity
  requirement from Assumption \ref{assumption-cont} is automatically
  satisfied.  
  
  Let $(\rho,\cH)$ be a unitary $\sigma$-representation of $G$,
  and let $\rho_\sigma$ be the Mackey representation of $G_\sigma$.
  Then $(\rho_\sigma,\cH)$ is a (strongly) continuous representation 
  (follows from Example \ref{ex:hilbertcontinuity}
  and Theorem \ref{thm:weakcontinuity}).
  Moreover, assume that $\rho$ is an irreducible square integrable
  projective representation and assume there is a non-zero
  $u$ for which $(u,\rho(\cdot)u)$ is in $L^1_w(G)$ for some
  submultiplicative weight $w\geq 1$, see Remark \ref{rem:leftandright} bellow for definition.
  Let $\cS$ be the Banach space 
  $$
  \cS = \cH_w^1 = \{v\in \cH \mid (v,\rho(\cdot)u)\in L^1_w(G) \}
  $$
  equipped with the norm it inherits from $L^1_w(G)$.
  It is clear that for a cocycle $\sigma$ for 
  $\rho$ the space $\cH_w^1$ is isometrically isomorphic to the space 
  $$
  \{v\in \cH \mid (v,\rho_\sigma(\cdot)u)\in L^1_w(G_\sigma) \}
  $$
  when $w(x,t)=w(x)$.

 {}From standard coorbit theory \cite{Feichtinger1989}
  the representation $\rho_\sigma$ restricted to $H_w^1$ 
  is strongly continous (because left translation is continous
  on $L^1_w(G_\sigma)$). Therefore it follows immediately that
  $(x,t)\mapsto \langle \lambda,\rho_\sigma(x,t)v\rangle$ is
  continuous on $G_\sigma$ for $\lambda\in (H_w^1)^*$
  and $v\in H_w^1$ . Then from Theorem \ref{thm:weakcontinuity}
  it follows that $x\mapsto \langle \lambda,\rho(x)v\rangle$
  is continuous on a neighbourhood of $e$.
 \end{example}

\begin{example}
  Assume that $G$ is a Lie group.  
  We will show that the smooth vectors for a projective
  representation satisfy the weak continuity from
  Assumption \ref{assumption-cont}.

  The smooth vectors for the projective representation $\rho$
  on the Hilbert space $\cH$ is the collection of vectors
  $u$ for which $G\ni x\mapsto \rho(x)u\in\cH$ 
  is $C^\infty$ on a neighbourhood of $e$.
  By arguments similar to those in \cite{Poulsen1972} this is equivalent to
  the weak smoothness of the mapping $x\mapsto (v,\rho(x)u)$
  for any $v\in \cH$ on a neighbourhood of $e$. 
  By Theorem \ref{thm:weakcontinuity} we then get that
  $(x,t) \mapsto (v,\rho_\sigma (x,t)u)$ is smooth, which tells us
  that $u$ is a smooth vector for $\rho_\sigma $ by \cite{Poulsen1972}.
  This argument also works in the reverse direction, so
  we see that the smooth vectors for $\rho$ are the same 
  as the smooth vectors for $\rho_\sigma$, i.e. 
  $\cH_\rho^\infty=\cH_{\rho_\sigma}^\infty$. 
  Since $ \rho_\sigma (x,t)u = t\rho(x)u$ the derivatives in $t$ 
  are just multiples of the identity. Therefore the usual
  Fr\'echet space topologies on $\cH_\rho^\infty$ and $\cH_{\rho_\sigma}^\infty$
  are generated by the same differential operators,
  and therefore their topologies are equivalent.
  Therefore they 
  also have the same dual spaces $\cH_\rho^{-\infty}=\cH_{\rho_\sigma}^{-\infty}$.

  It is classical, see \cite{Warner1972}, that 
  $ \rho_\sigma$ restricted to $\cH_{ \rho_\sigma}$ is a continuous representation.
  Therefore $(x,t) \mapsto \langle \lambda, \rho_\sigma (x,t)u\rangle$
  is continuous on $G_\sigma$ for $\lambda\in \cH_{ \rho_\sigma }^{-\infty}$
  and $u\in \cH_{ \rho_\sigma }^{\infty}$. Therefore Theorem \ref{thm:weakcontinuity}
  ensures that $x \mapsto \langle \lambda,\rho(x)u\rangle$
  is continuous on a neighbourhood of $e$.
\end{example}

Let $u\in\fS$ be $\rho$-cyclic vector. We define the wavelet transform
$W^\rho_u:\fS^*\to C(G)$ by
$$W_u^\rho(\lambda)(x):=\langle \lambda,\rho(x)u\rangle.$$ 
In the following lemma we state the relation between a projective
representation and the corresponding representation of the Mackey
group. The proof is straightforward from the definition.
\begin{lemma} \label{lemma:waveletrelation} Let $(\rho,\fS)$ be a
  projective representation of $G$ and let $(\rho_\sigma,\fS)$ be the
  corresponding representation of $G_\sigma$. Then the following are
  true:
  \begin{enumerate}
  \item The vector $u\in \fS$ is $\rho$-cyclic \Iff \, $u$ is
    $\rho_\sigma$-cyclic.
  \item The wavelet transforms are related by
    $W_u^{ \rho_\sigma }(\lambda)(x,t)=\overline{t}W^\rho_u(\lambda)(x).$
  \end{enumerate}
\end{lemma}

\section{Banach function spaces and sequence spaces}
\noindent
In this section we define twisted translation and twisted convolution,
and we summarize the assumptions we will place on a Banach space of
functions (or BF-space for short) throughout this paper.  We will also
introduce a collection of sequence spaces which will be used in
Sections \ref{sec:atomicdecomposition} and \ref{sec:bergmanspaces} in
order to formulate our results on atomic decompositions.

\begin{definition}
  Let $B$ be a Banach space over $\C$ of functions on $G$, and let
  $\sigma$ be a cocycle on $G$.  For a function $ f \in B$, we define
  twisted left translation by
$$ \ell_y^\sigma f(x):=\overline {\sigma(y, y^{-1}x)} f(y^{-1}x),$$ and  
we define twisted right translation by
$$ r_y^\sigma f(x):= {\sigma(x,y)} f(xy).$$

If the cocycle is $1$ we set $\ell_y = \ell^\sigma_y$ and
$r_y=r^\sigma_y$, and we drop the use of the word twisted.
\end{definition}

It is important to notice the following relations between the
translation operators and the wavelet transform
\begin{align*}
  \ell^\sigma_y W^\rho_u(\lambda)(x) &= W^\rho_u(\rho^*(y)\lambda)(x) \\
  r^\sigma_y W^\rho_u(\lambda)(x) &= W^\rho_{\rho(y)u}(\lambda)(x).
\end{align*}

We say that the BF-space $B$ is twisted left-invariant if
$\ell^\sigma_y f \in B $ for all $f \in B$ and
$f \mapsto \ell^\sigma_y f $ is bounded for all $y \in G$.
Analogously, we define twisted right-invariant spaces.  In the sequel
we will assume any BF-space to be twisted right- and left-invariant,
and that twisted left and right translations by elements $y$ in a
compact set $U$ are uniformly bounded in the sense that there is a
finite constant $C_U$ such that for any $y\in U$
\begin{equation}
  \label{eq:uniformboundrep}
  \| \ell^\sigma_y f\| \leq C_U \| f\|
  \text{ and }
  \| r^\sigma_y f\| \leq C_U \| f\|.
\end{equation}
We say that left translation, $\ell$, is continuous on $B$ if
for every $f\in B$ the mapping $y \mapsto \ell_y f$ is continous
$G\to B$. Right translation $r$ is defined to be continous in a similar manner.
In general we do not assume that
$\ell$ and $r$ are continuous on $B$ until we need to derive
atomic decompositions in sections \ref{sec:atomicdecomposition} and
\ref{sec:bergmanspaces}. Moreover, the twisted translations
$\ell^\sigma$ and $r^\sigma$ are generally not
continuous (not even near $e$), since $\sigma$ is not globally continuous.
A BF-space $B$ is called solid if
$|f|\leq |g|$ and $g\in B$ implies that $f\in B$.  If $B$ is a solid
space, then $B$ is twisted left or right invariant if and only if $B$
is left or right invariant.

\begin{remark}\label{rem:leftandright}
  A weight on $G$ is a measurable 
  function $w:G\to (0,\infty)$.  For a weight $w$
  we define the weighted Lebesgue spaces
  $$
  L^p_w(G) = \left\{ \text{$f$ measurable} \,\Big|\, \| f\|_{L^p_w}
    :=\left( \int |f(x)|^p w(x)\,dx \right)^{1/p} <\infty \right\}.
  $$
  The spaces $L^p_w(G)$ are clearly solid Banach function spaces.
  
  The weight $w$ is called submultiplicative, if $w(xy) \leq w(x)w(y)$
  for all $x,y\in G$.  When $w$ is submultiplicative the spaces
  $L^p_w(G)$ are left and right invariant for $1\leq p \leq \infty$ and left
  and right translation are continuous for $1\leq p<\infty$.
  Moreover, $\ell^\sigma$ and $r^\sigma$ are projective
  representations on $L^p_w(G)$ for $1\leq p<\infty$.  The cocycle for
  $\ell^\sigma$ is $\sigma$ while $r^\sigma$ has cocycle
  $\overline{\sigma}$.

  It is easily verified that if $w$ is bounded on bounded sets, then
  condition 
  \eqref{eq:uniformboundrep} holds on $L^p_w(G)$ for $1\leq p<\infty$.
  By direct inspection of the norm, it can also be verified for
  $p=\infty$.
\end{remark}

We now define sequence spaces related to a solid BF-space $B$.  Let
$U$ be a compact neighbourhood of the identity, and let $\{ x_i\}$ be
a countable $U$-dense and well-spread collection of elements in
$G$. Define the sequence space $\dot{b}$ by
$$
\dot{b} = \left\{ \{ c_i\}\subseteq \mathbb{C} \, \Big|\, \sum_i c_i
  1_{x_iU} \in B \right\},
$$
with norm
$$
\| \{ c_i\}\|_{\dot{b}} = \left\| \sum_i c_i 1_{x_iU} \right\|_B.
$$
For example, if $B=L^p(G)$, then $\dot{b}=\ell^p$.

In the remainder of this paper we will need the following
generalization of convolution in the presence of a cocycle.  For a
cocycle $\sigma$ we define the twisted convolution of functions $f$
and $g$ on $G$ by
$$f \# g(x):= \int\limits_G f(y) \ell_y^\sigma g(x) \, dy 
= \int\limits_G f(y) g(y^{-1}x)\overline{\sigma(y,y^{-1}x)} \, dy $$
whenever the integral exists.  If the cocycle is constant this is the same as usual group convolution
for which we reserve the special notation
$$
f *g(x):= \int\limits_G f(y) g(y^{-1}x) \, dy .
$$

  \begin{remark}
    \label{rem:weakconv}
    The twisted convolution always exists if $f,g$ are in $L^1$ or 
    if the product $f(y)g(y^{-1}x)$ is integrable in $y$.
    But sometimes it is necessary to consider more
    general cases. For example, twisted convolution might have to be
    defined weakly in the following way.
    If
    $0\leq \psi_n\leq 1$ is an increasing
    sequence of compactly supported continuous
    functions which are identically $1$ on nested compact sets $C_n$
    satisfying $\cup_n C_n = G$, then we might define the
    twisted convolution by
    \begin{equation}
      \label{eq:weakdefconv}
      f \# g(x)= \lim_{n\to\infty}
      \int_G \psi_n(y)f(y) \ell_y^\sigma g(x) \, dy 
    \end{equation}
    whenever the limit exists.
    In applications one has to verify that such a definition makes sense.
    Notice, that if we know that the product $f(y)g(y^{-1}x)$ is
    integrable in $y$, then Lebesgue's Dominated Convergence Theorem tells
    us that this weak definition agrees with
    $$f \# g(x):= \int\limits_G f(y) \ell_y^\sigma g(x) \, dy      . $$
  \end{remark}

\begin{example}[Locally Compact Abelian Groups]
 The best known example for twisted convolution is $\R^{2n}$ and its relation to time-frequency analysis and the
 Schr\"odinger representation of the Heisenberg group. This example
can be generalized to locally compact abelian groups, see \cite{Grochenig1998, 
Feichtinger2003a,Sielemann2016,Sielemann2016a,Matusiak2007} and
the references therein. For this we assume that $G=H\times \hH$ where $H$ is a locally compact abelian group and
$\hH$ is the dual group of continuous homomorphisms 
$\varphi : H \to \T$. The topology is the product topology and the 
product is defined as the product of each of the components. Define
\[\sigma ((x,\varphi), (y,\psi ))=\overline{\varphi  (y)}\, .\]
It is easy to see that $\sigma $ is a cocycle. The following
time-frequency representation is then an example of a $\sigma$-representation with continuous cocycle.
Define translation by $T_xf(y)=f(x^{-1}y)$ and modulation by
$M_\varphi f(y)=\varphi (y)f(y)$. Then
\[T_xM_\varphi  = \overline{\varphi (x)}M_\varphi T_x. \]
Thus, with $\rho (x,\varphi)=T_xM_\varphi$ we get
a projective representation, since
\begin{align*}
\rho (xy,\psi\varphi) &= T_{xy}M_{\psi \varphi}\\
&= T_xT_yM_\psi M_\varphi\\
&= \overline{\psi (y)}  T_xM_\psi T_y M_\varphi\\
&=\sigma ((x,\psi ),(y,\varphi )) \rho (x,\psi )\rho(y,\varphi)\, .
\end{align*}
The twisted convolution now becomes a simple generalization of the well known twisted convolution
$\R^n$:
\[ f \# g((x, \psi) )=  \int_G f((y,\phi ) g((y^{-1}x, \bar \phi \psi )) \psi (y)\overline{\psi (x)}  \, dyd\psi \]

We would like to point to the reference \cite{Feichtinger2003a} 
where modulation spaces are defined for Abelian groups. 
In the special case where $H$ is an abelian Lie group, then $\hH$ is also a Lie group, which 
could be discrete and countably infinite. Thus $H\times \hH$ is also
a Lie group and the results of this article, in particular 
the discretization, become available for the modulation spaces. 
\end{example}

\section{Coorbit Spaces}
\noindent
Let $ u\in \fS $ and let $B$ be a twisted left-invariant BF-space.  Define
$$\coBP:=\{ \lambda \in \fS^*\mid  W^\rho_u(\lambda) \in B \}  $$
with the norm
$$\|\lambda\|_{\coBP} := \|W^\rho_u(\lambda)\|_B.$$
We will now impose conditions which ensure that $\coBP$ is a Banach
space.  In the process we will demonstrate that these conditions
ensure that that the space
$$ B_u^\#:=\{f \in B \mid f \# W^\rho_u(u)=f \} $$
with norm inherited from $B$ is a reproducing kernel Banach space
isometrically isomorphic to $\coBP$.

A $\rho$-cyclic vector $u\in \fS$, is called a $\rho$-analyzing vector
for $\fS$ if the reproducing formula
$$W^\rho_u(\lambda)\#W^\rho_u(u)=W^\rho_u(\lambda)$$
holds for all $\lambda\in \fS^*$.

\begin{assumption} \label{assumptionrhocoorbit} Let $B$ be a twisted
  left-invariant BF-space on $G$.  Assume there exists a nonzero
  $\rho$-analyzing vector $u \in \fS$ satisfying the following
  continuity condition: The mapping
$$ B\times\fS \ni (f,v)\mapsto f\#W^\rho_v(u)(1)=\int_G f(y)W^\rho_v(u)^\vee(y)\overline{\sigma(y,y^{-1})}  \, \,d y \in \C $$  is continuous.
\end{assumption}

  \begin{remark}
    The twisted convolution
    $W^\rho_u(\lambda)\#W^\rho_u(u)$ for $\lambda \in \fS^*$ might only
    be defined in a weak sense as mentioned in Remark~\ref{rem:weakconv}.
    The conditions we have placed on the Banach space $B$ ensure that
    the twisted convolution $f\# W^\rho_u(u)$ exists as an integral
    for all $f\in B$.
    Therefore, for $\lambda \in \coBP$,
    the two definitions agree,
    since $W_u^\rho(\lambda)\in B$.
    This means that from this point on, whenever $\lambda\in\coBP$
    the twisted convolution can and should be interpreted as an integral.
    
    This observation is essential for
    producing atomic decompositions in section \ref{sec:atomicdecomposition}.
  \end{remark}

\begin{remark}\label{Rem:assumptionrhocoorbit}
  If $B=L^p_w(G)$ then the continuity condition will be a duality
  requirement. Specifically, Assumption \ref{assumptionrhocoorbit} is
  satisfied for $B=L_w^p(G)$ if the topology on $\fS$ is such that
 $$\fS \ni v\mapsto W^\rho_v(u)^{\vee} \in L^q_{{w^{-q/p}}}(G)$$ is continuous, where $\frac{1}{p}+\frac{1}{q}=1$. 
\end{remark}

The main result of this section is
\begin{theorem}\label{coorbitthm}
  Let $(\rho,\fS)$ be a projective representation of $G$, and let $B$
  be a twisted left-invariant BF-space on $G$.  Assume that $u\in \fS$
  is a $\rho$-analyzing vector satisfying Assumption
  \ref{assumptionrhocoorbit}.  Then
  \begin{enumerate}
  \item $W^\rho_u(v)\#W^\rho_u(u)=W^\rho_u(v)$ for $v \in \coBP$.
  \item The space $\coBP$ is a $\rho^*$-invariant Banach space.
  \item $W^\rho_u: \coBP \to B$ intertwines $\rho^*$ and
    $\ell^\sigma$.
  \item $W^\rho_u: \coBP \to B_u^\#$ is an isometric isomorphism.
  \item $\coBP = \{\rho^*(F)u \mid F\in B_u^\# \}$ when
      $\rho^*(F)u$ is defined by
      $\langle \rho^*(F)u,v \rangle = \int F(x) \langle \rho^*(x)u,v\rangle\,dx$.
  \end{enumerate}
\end{theorem}

From \cite{Christensen2011} it is known that the statements are true if the
cocycle $\sigma$ is constant and therefore $\rho$ is a representation.
Notice that in \cite{Christensen2011} the representation $\rho$ is
assumed strongly continuous, but that requirement can be replaced by
the weak continuity of $x\mapsto \langle \lambda,\rho(x)u\rangle$
for $\lambda\in\cS^*$ and $u\in\cS$ without modifications.
We will therefore prove Theorem \ref{coorbitthm} by
connecting it to  
coorbit theory for the representation 
$\rho_\sigma$ for an appropriate choice of BF-space $\widehat{B}$ on the
Mackey group.  It turns out that the space
$$\Bh:= \{ F : G\times\T\to \mathbb{C}\,\,\mid F(a,t)= \overline{t}f(a), f \in B\}$$ 
with norm $\| F \| _{\widehat{ B}}:=\| f \|_B$ is a good choice.

\begin{lemma} \label{propertiesbhat} If $G\times\T $ and $\Bh$ are
  defined as before, then the following relations hold.
  \begin{enumerate}
  \item The spaces $B$, $\Bh$ are isometrically isomorphic via
    $\Lambda f(x,t):= \overline{t}f(x)$.
  \item If the space $B$ is twisted left-invariant, then $\Bh$ is
    left-invariant.
  \item For $F \in \Bh$, we have
    $F*W^{\rho_\sigma}_u(u)(x,z)=\overline{z}\,f\,\#\,W_u^\rho(u)(x)$
    when $F(a,t)=\overline{t}f(a)$.
  \end{enumerate}
\end{lemma}
 
\begin{proof}
  The first part is clear. 
  The second part follows from the following calculations:
  \begin{align*}
    \ell_{(a,w)}F(x,z)= & F(a^{-1}x,\overline{w}z \sigma(a,a^{-1})\overline{ \sigma(a^{-1},x)})\\
    = & F(a^{-1}x,\overline{w}z \sigma(a,a^{-1}x)) \\
    = & w \overline{z} \,\overline{\sigma(a,a^{-1}x)}  f(a^{-1}x)\\
    = & \overline{z}w \,\ell^\sigma_a f(x). 
  \end{align*} 
  Therefore, $B$ is twisted left-invariant $\Iff$ $\Bh$ is left
  invariant. 
  
  For the third part, we have
  \begin{align*}
    F*W_u(u)(x,z)
    =& \iint  F(y,w)W_u^{\rho_\sigma}(u)((y,w)^{-1}(x,z))dw dy\\
    =&\iint  F(y,w)W_u^{\rho_\sigma}(u)(y^{-1}x,\overline{w}z\sigma(y,y^{-1})\overline{\sigma(y^{-1},x)})dw dy\\
    =&\overline{z}\int  f(y)W^\rho_u(u)(y^{-1}x) \overline{\sigma(y,y^{-1})}\sigma(y^{-1},x)\,d y\\
    =&\overline{z}\,f\# W^\rho_u(u)(x).\qedhere
  \end{align*}
\end{proof}

\begin{lemma}
  A vector $u$ is $\rho$-analyzing if and only if u is
  $\rho_\sigma$-analyzing.
\end{lemma}
\begin{proof}
  First, we know that $u$ is $\rho_\sigma$-cyclic if and only if $u$ is
  $\rho$-cyclic, since
  $\langle \lambda,\rho(x)u\rangle = t \langle
  \lambda,\rho_\sigma(x,t)u\rangle$.

  A straightforward calculation gives
$$
W^{\rho_\sigma}_u(\lambda)*W^{\rho_\sigma}_u(u)(x,t) = \overline{t}
W^{\rho}_u(\lambda)\# W^{\rho}_u(u)(x),
$$
and the claim follows from this.
\end{proof}

The following theorem provides the connection between the coorbit
theory that arises from representations \cite{Christensen2011} and the coorbit
theory that arises from projective representations.
\begin{theorem} \label{rhovspiassumption} 
  The triple $B$, $\rho$ and $u$ satisfy Assumption
  \ref{assumptionrhocoorbit} with cocycle $\sigma$ if and only if the
  triple $\widehat{B}$, $\rho_\sigma$ and $u$ satisfy Assumption
  \ref{assumptionrhocoorbit} with a constant cocycle.  Therefore
  $\coBP$ and $\coBh$ are simultaneously defined. Moreover,
  $\coBP =\coBh$ with the same norm.
\end{theorem}

\begin{proof}
  By Lemma~\ref{propertiesbhat} the space $\Bh $ is left invariant.
  Next, denote the wavelet transform related to the representation
  $\rho_\sigma$ by $W^{\rho_\sigma}_u$.

  Note that if $F(x,t)=\overline{t}f(x)$, then
  \begin{align*}
    \iint  F(x,z)W^{\rho_\sigma}_v(u)((x,z)^{-1})  \,dz dx 
    &=\iint  F(x,z)W^{\rho_\sigma}_v(u)((x^{-1},\overline{z}\sigma(x^{-1},x) ))  \,dz dx \\
    &=\iint \overline{z} f(x) W^{\rho}_v(u)(x^{-1}) z \overline{\sigma(x,x^{-1})} \,dzdx \\
    &= f\#W^\rho_v(u)(1).
  \end{align*}
  It follows that the continuity of
  $(f,v)\mapsto \int f(x)W^\rho_v(u)(x^{-1})
  \overline{\sigma(x,x^{-1})} \, dx$
  on $B\times\fS$ is equivalent to the continuity of
  $(F,v)\mapsto \iint F(x,z)W^{\rho_\sigma}_v(u)((x,z)^{-1}) \,dzdx$.

  Next, assume that $\lambda\in \fS$. Then, by
  Lemma~\ref{lemma:waveletrelation}, $\lambda\in \coBP$
  $\Leftrightarrow$ $W^\rho_u(\lambda)\in B$ $\Leftrightarrow$
  $W_u^{\rho_\sigma}(\lambda)\in \Bh$ $\Leftrightarrow $
  $\lambda\in \coBh$.
\end{proof}

Now we demonstrate our main result about the coorbit space constructed
by the twisted convolution.

\begin{proof}[Proof of Theorem~\ref{coorbitthm}]
  By Theorem \ref{rhovspiassumption}, the space $\Bh$ and u satisfy
  Assumption \ref{assumptionrhocoorbit}. So we can apply Theorem
  \ref{rhovspiassumption} to the space $\Bh$.

  (1) The identity is assumed true for all analyzing vectors $u$ and
  all functionals $\lambda$, and therefore it is also true for members
  of the coorbit space.

  (2) We know that the space $ \coBh = \coBP $ is
  $ \rho_\sigma^* $-invariant Banach space. So
  $ W_u(\rho_\sigma^*(y,w)\phi) \in \coBh $. On the other hand
$$ W^{\rho_\sigma}_u(\rho_\sigma^*(y,w)\phi)(x,z)
=\overline{z}w W^\rho_u(\rho^*(y)\phi)(x), $$
which implies that $W^\rho_u(\rho^*(y)\phi) \in B$.

(3) Using the fact that $ W^{\rho_\sigma}_u$ intertwines $\rho_\sigma^*$
with left translation, and $ \rho_\sigma^*(x,z)=z\rho^*(x) $. We have
\begin{align*}
  W^\rho_u(\rho^*(y)\phi)(x)
  =& \overline{w}z W^{\rho_\sigma}_u(\rho_\sigma^*(y,w)\phi)(x,z)
     = \overline{w}z \ell_{(y,w)}W^{\rho_\sigma}_u(\phi)(x,z)\\
  =& \overline{\sigma(y,y^{-1})}\sigma(y^{-1},x) \ell_yW^\rho_u(\phi)(x)\\
  =& \ell^\sigma_yW^\rho_u(\phi)(x) .
\end{align*}

(4) According to (1) and (3) in Lemma~\ref{propertiesbhat} the spaces
$B^\#_u$ and
$$
\widehat{B}_u = \{F\in \widehat{B} \mid F*W^{\rho_\sigma}_u(u) = F \},
$$
where convolution is on the group $G_\sigma$, are isometrically
isomorphic.  If we denote the isometrical isomorphism between $B_u$
and $\widehat{B}_u$ by $\Lambda$, then
$W^\rho_u = \Lambda ^{-1} W^{\rho_\sigma}_u : \coBP \to B_u^\#$ and the
result is obtained.

(5)
First, $\rho^*(F)u$ is well-defined due to Assumption~\ref{assumptionrhocoorbit}.
Also, if $F\in B_u^\#$, then
 $$
 W_u^\rho(\rho^*(F)u)(x) =
 \langle \rho^*(F)u,\rho(x)u  \rangle
 = F\# W_u^\rho(u)(x)
 = F(x).
 $$
 This shows that $\rho^*(F)u$ is in $\coBP$.
 If, on the other hand, $v\in \coBP$, then $W_u^\rho(v)$ is
 in $B$. Then $W_u^\rho(v)$ satisfies the
 twisted convolution reproducing formula 
 $W_u^\rho(v)\# W_u^\rho(u) = W_u^\rho(v)$,
 which means
 $$
 \langle v,\rho(x)u \rangle
 = \int W_u^\rho(v)(y) \langle \rho^*(y)u,\rho(x)u \rangle \,dy
 = \langle \rho^*(F)u,\rho(x)u \rangle.
 $$
 Since $u$ is cyclic, we see that $v=\rho^*(F)u$.
 \end{proof}

In the following theorem, we prove that the twisted coorbit space is
independent of the choice of the $\rho$-analyzing vector under some
assumptions.
\begin{theorem}\label{thm:u-indep}
  Assume that $u_1$ and
  $u_2$ 
  both satisfy Assumption \ref{assumptionrhocoorbit}, and the
  following properties are true for $i,j\in \{1,2\}$
  \begin{enumerate}
  \item there are nonzero constants $C_{i,j}$ such that
    $W^\rho_{u_i}(\lambda)\#W^\rho_{u_j}(u_i)=C_{i,j}W^\rho_{u_j}(\lambda)$
    for all $\lambda \in \fS^*$
  \item the mapping $B_{u_i}\ni f\mapsto f\# W^\rho_{u_j}(u_i)\in B$
    is continuous.
  \end{enumerate}
  Then $\Co_\rho ^{u_1}B=\Co_\rho ^{u_2}B$ with equivalent norms.
\end{theorem}
\begin{proof}
  We already know from \cite{Christensen2011} 
  that this theorem is true for 
  representations, i.e. when $\sigma=1$. The proof thus consists of
  connecting the statements for $\rho$ with similar statements for
  $\rho_\sigma$.

  Consider the space $\Bh$ and the Mackey group $G\times \T$. Since
  $u_1$ and $u_2$ are $\rho$-analyzing vectors for $\fS$ that satisfy
  Assumption \ref{assumptionrhocoorbit}, they are also
  $\pi_{\rho}$-analyzing vectors for $\fS$ that satisfy Assumption
  \ref{assumptionrhocoorbit} (see Theorem
  \ref{rhovspiassumption}). Also for $i,j\in\{1,2\}$ and
  $\lambda\in \fS^*$, we have
  \begin{align*}
    W^{\rho_\sigma}_{u_i}(\lambda)* W^{\rho_\sigma}_{u_j}(u_i)(x,t)=
    &\overline{t}W^\rho_{u_i}(\lambda)\#W^\rho_{u_j}(u_i)(x)\\
    =&\overline{t}C_{i,j}W^\rho_{u_j}(\lambda)(x)\\
    =&C_{i,j}W^{\rho_\sigma}_{u_j}(\lambda)(x,t)
  \end{align*}
  Moreover, the mapping
  $\Bh_{u_i}\ni F\mapsto F*W^{\rho_\sigma}_{uj}(u_i)\in \Bh$ is
  continuous, indeed,
  \[\|F *W^{\rho_\sigma}_{uj}(u_i)\|_{\widehat{B}}=\|f\#W^\rho_{u_j}(u_i)\|_B\leq
  C\|f\|_B=C\|F\|_{\Bh}\, .\]
  Therefore, by Theorem 2.7 in \cite{Christensen2011},
  $\Co_{\rho_\sigma}^{u_1}\Bh=\Co_{\rho_\sigma}^{u_2}\Bh$.  Since
  $\Co_{\rho_\sigma}^{u_i}\Bh=\Co_\rho^{u_i}B$, the result is obtained.
\end{proof}

\section{Atomic decompositions and frames}
\label{sec:atomicdecomposition}
\noindent
The atomic decompositions we provide build on sampling on the
reproducing kernel space $B_u^\#$.  To do so, we need to show that the
reproducing kernel does not vary too much under translation by
elements of a fixed compact neighbourhood.  The special form of the
reproducing kernel allow us to estimate such local variations using
smoothness of the analyzing vector $u$. Assume that $G$ is a connected
  Lie group with Lie algebra $\mathfrak{g}$, and that $\rho$ is a projective
  representation of $G$ on the Fr\'echet space $\fS$.

\begin{definition}
  A vector $u\in\fS$ is called $\rho$-weakly differentiable if for all
  $X\in\fg$ and for all $\lambda \in \fS^*$ the mapping
  $t\mapsto \langle \lambda,\rho(\exp(tX))u \rangle$ is
  differentiable and there is a $u_X\in \fS$ satisfying
  $$ 
  \langle \lambda,u_X \rangle = \frac{d}{dt}\Big|_{t=0} \langle
  \lambda,\rho(\exp(tX))u \rangle.
  $$
We then define $\rho(X)u=u_X$.
If for any $\lambda\in \fS$ 
    the mapping $x\mapsto \langle \lambda,\rho(x)u\rangle$ is
    differentiable up to order $n$, we say that the vector $u$
    is $\rho$-weakly differentiable up to order $n$.

  Similarly, a vector $\lambda\in\fS$ is called $\rho^*$-weakly
  differentiable if for all $X\in\fg$ and for all $u \in \fS$ the
  mapping $t\mapsto \langle \rho^*(\exp(tX))\lambda, u \rangle$ is
  differentiable. In this case we denote by $\rho^*(X)\lambda$ the
  distribution in $\fS^*$ satisfying
  $$ 
  \langle \rho^*(X)\lambda,u \rangle = \frac{d}{dt}\Big|_{t=0} \langle
  \rho^*(\exp(tX))\lambda, u \rangle.
  $$
 If for any $u\in \mathcal{S}$ 
    the mapping $x\mapsto \langle \rho^*(x)\lambda,u\rangle$ is
    differentiable up to order $n$, we say that the vector $\lambda$
    is $\rho^*$-weakly differentiable up to order $n$.
\end{definition}

\begin{remark} As pointed out by one of the referees,
    the existence of $\rho(X)u$ in the definition will be
    guaranteed by the Banach-Steinhaus theorem
    if the space $\fS$ is quasi-reflexive (and hence $\fS^*$ is barreled).
    We often work with smooth vectors for the projective representation $\rho$
    in which case the mapping $x\mapsto \rho(x)u$ is strongly differentiable.
   Therefore the existence of $\rho(X)u$ is guaranteed for most
   of our applications.

    Remembering that $\fS^*$ is equipped by the weak* topology, we
    know that $(\fS^*)^*=\fS$ is barreled and therefore the existence
    of $\rho^*(X)u$ is automatic. 
\end{remark}

\begin{assumption}
  \label{assumptionatomic}
  Let $B$ be a solid BF-space on $G$ on which left and right translation are
  continuous, and assume that $B$, $\rho$ and $u$ satisfy
  Assumption~\ref{assumptionrhocoorbit}.  Assume, moreover, that $u$
  is $\rho$-weakly and $\rho^*$-weakly differentiable up to order
  $n=\dim(G)$, and that for all finite subsets with $n$ elements,
  $\{ Y_1,Y_2,\dots,Y_n\}\subseteq \fg$ the mappings
  $$
  f\mapsto f*|W^\rho_u(\rho^*(Y_1)\rho^*(Y_2)\cdots\rho^*(Y_n)u)|
  \text{ and } f\mapsto
  f*|W^\rho_{\rho(Y_1)\rho(Y_2)\cdots\rho(Y_n)u}(u)|
  $$
  are bounded on the solid, left and right invariant BF-space $B$.
\end{assumption}

Notice, that all convolutions in this section and in
Appendix~\ref{appendixA} are expected to be
defined as proper integrals. We are no longer allowing weak
definitions as in Remark~\ref{rem:weakconv}, or at least we have to show
that the weak definition agrees with an integral.

\begin{theorem}
  \label{Thm:atomicdecomp}
  Let $u$ be a vector satisfying
  Assumption~\ref{assumptionatomic}.  Let $X_1,\dots,X_n$ be a fixed
  basis for $\fg$ and define
  $U_\epsilon=\{ e^{t_1X_1}\cdots e^{t_nX_n} \mid -\epsilon \leq t_k
  \leq \epsilon\}$. In the following we always choose 
  a cocycle $\sigma$ and 
  $\epsilon>0$ small enough so that $\sigma$ is $C^\infty$
  on a neighbourhood containing $U_\epsilon\times U_\epsilon$.
  \begin{enumerate}
  \item 

  Given a $U_\epsilon$-dense and well-spread sequence
  $\{ x_i\}\subseteq G$ and a $U_\epsilon$-BUPU $\{\psi_i \}$,
  the operators
  \begin{align*}
    T_1 f &= \sum_i f(x_i) \sigma(x,x^{-1}x_i)\psi_i\# W^\rho_u(u) \\
    T_2 f &= \sum_i \lambda_i(f) \ell_{x_i}^\sigma W^\rho_u(u) \\
    T_3 f &= \sum_i c_i f(x_i) \ell_{x_i}^\sigma W^\rho_u(u)
  \end{align*}
  where
  $\lambda_i(f) = \int
  f(y)\psi_i(y)\overline{\sigma(y,y^{-1}x_i)}\,dy$
  and $c_i = \int \psi_i(y)\,dy$, are well defined 
  from $B_u^\sigma$ to $B_u^\sigma$.
\item There is an $\epsilon$ small enough for which the operators
  $T_1,T_2,T_3$ are invertible for any $U_\epsilon$-dense and
  well-spread sequence $\{ x_i\}\subseteq G$ and any $U_\epsilon$-BUPU
  $\{\psi_i \}$. In this case
  the family $\{\rho^*(x_i)u\}$ is a Banach frame for
  $\coBP$ with respect to the sequence space $\dot{b}$, and the
  families
  $\{ \lambda_i \circ T^{-1}_2 \circ W^\rho_u\, , \rho^*(x_i)u\} $ and
  $\{ c_iT^{-1}_3 \circ W^\rho_u\, , \rho^*(x_i)u\} $ are atomic
  decompositions for $\coBP$ with respect to the sequence space
  $\dot{b}$. In particular, $\gamma\in \coBP$ can be reconstructed by
  \begin{align*}
            \gamma &= (W^\rho_u)^{-1} T^{-1}_1 \Big(\sum_i W^\rho_u(\gamma)(x_i) \psi _i \# W^\rho_u(u)\Big)\\
            \gamma &= \sum_i \lambda_i\big(T^{-1}_2 W^\rho_u(\gamma)\big)\rho^*(x_i)u\\  
            \gamma &= \sum_i c_i\,\,(T^{-1}_3 W^\rho_u(\gamma))\,\rho^*(x_i)u 
          \end{align*}
          with convergence in $\fS^*$. The convergence is in $\coBP$ if
          $C_c(G)$ is dense in $B$.
\end{enumerate}
\end{theorem}
 
These results can be proven by applying Theorems 3.4, 3.6 and 3.7
from \cite{Christensen2012} to the representation $\rho_\sigma$ 
of the Mackey group $G_\sigma$. 
This requires us to 
define an appropriate solid BF-space on $G_\sigma$. There
are many more or less natural choices for such BF-spaces and two
different approaches have been presented in \cite{Christensen1996} and
\cite{Darweesh2015}. The more natural choice is found in
\cite{Darweesh2015} where a minimal extension, naturally isomorphic to the original choice, is considered. The
drawback is, that the space is not solid.
In this paper we wish to avoid making this choice
and to work directly on the group $G$ and the Banach 
space $B$.
One of the 
  advantages is that the atoms are obtained via sampling at points
  in the group $G$ and in a reproducing kernel subspace
  of $B$.
  This makes the apprach presented here natural.
  In \cite{Christensen1996,Darweesh2015} the atoms
  are instead obtained from sampling at points in the extended group $G_\sigma$
  on an reproducing kernel subspace of functions
  on $G_\sigma$. 
  Moreover, the other two approaches
  rely on the topology of $G_\sigma$ being the product
  topology of $G\times\T$,
  a property which can only be ensured if the multiplier is continuous.
  This property is utilized when choosing the sample points in $G_\sigma$
  (see for example the proof of
  Theorem 6.1 on p. 1305 of \cite{Christensen1996}). The present approach
  avoids this assumption on $\sigma$.
The calculations differ
non-trivially from \cite{Christensen2012} by the occurrence of the cocycle, and the
details are carried out in the appendix.

The result above is focused on using differentiable vectors
  as atoms. It should be mentioned that this
  is not strictly necessary, since
  the results in the appendix can be used for vectors that are
  not necessarily differentiable. It is possible to obtain
  atoms for non-differentiable vectors, as long as it can be shown
  that convolution with local oscillations of the kernel $W_u^\rho(u)$
  is bounded on the space $B$ (see Corollary~\ref{cor:convosc}) . We intentionally left this to the appendix,
  since smooth vectors suffice for our intended application.

\section{Bergman Spaces on the Unit Ball}\label{sec:bergmanspaces}
\noindent
Recently, in \cite{Christensen2016}, the first and the third authors together with  
K. Gr\"ochenig,  obtained atomic decompositions for
Bergman spaces on the unit ball
in $\mathbb{C}^n$ through a finite covering group of the group $\SU(n,1)$ with
the restriction that the representation parameter $s>n$ had to be rational. 
As a special case, atomic decompositions of Bergman spaces
through the group $\SU(n,1)$ is valid for integer values of the
parameter $s>n$. Overcoming the restriction on the parameter
was one reason for introducing a coorbit theory for
projective representations.

We dedicate this chapter to generating Banach frames and atomic
decompositions of Bergman spaces on the unit ball via the group
$\SU(n,1)$. 
For more references we encourage the
reader to see \cite{Chebli2004,Faraut1990,HarishChandra1955,Knapp1986,Wallach1979,Zhu2005}.

\subsection{Facts about Bergman Spaces on the Unit Ball}

In this section we collect facts about Bergman spaces on the unit
ball. Let $\mathbb{C}^n$ be equipped with the usual inner product
$(z,w)=z_1\overline{w_1}+z_2\overline{w_2}+...+z_n\overline{w_n}$ and
define the unit ball by
$$\mathbb{B}^n:=\left\{z\in \mathbb{C}^n \mid |z|^2:=
  |z_1|^2+|z_2|^2+...+|z_n|^2<1 \right\}.$$
Let $\Id v$ denote the normalized volume measure on the unit ball upon
identifying $\mathbb{C}^n$ with $\mathbb{R}^{2n}$.
 For $\alpha>-1$,
define the measure
$\Id v_\alpha (z):= C_\alpha (1-|z|^2)^\alpha \Id v(z)$, where
$C_\alpha=\frac{\Gamma(n+\alpha+1)}{n!\Gamma(\alpha+1)}$
makes $\Id v_\alpha$ a probability
measure.  Notice that the measure $\Id v_\alpha$ is finite measure on
$\mathbb{B}^n$ \Iff\, $\alpha >-1$.

We define the $\alpha$-weighted $L^p$ space on the unit ball
as
$$L^p_\alpha(\mathbb{B}^n)=\{f:\B\to \mathbb{C}\mid \int_{\B} |f(z)|^p
\Id v_\alpha(z)<\infty\}$$
with norm
$$\|f\|_{L^p_\alpha}=\left( \int_{\mathbb{B}^n} |f(z)|^p\, \Id
  v_\alpha(z)\right)^{1/p},$$
where $1\leq p<\infty$. For $\alpha>-1$, we define the weighted
Bergman spaces on the unit ball to be
$$A^p_\alpha(\B):= L^p_\alpha(\B)\cap \mathcal{O}(\B)$$
with norm inherited from $L^p_\alpha(\B)$, where $\mathcal{O}(\B)$ is
the space of holomorphic functions on the unit ball. We have the
condition $\alpha>-1$ to construct a non-trivial Bergman spaces, in
fact, if $\alpha\leq -1$, then the only holomorphic function in
$L^p_\alpha(\B)$ is the zero function.

As we have seen in the special case on the unit disc, Bergman spaces
are closed subspaces of $L^p_\alpha(\B)$, $i.e.$, Bergman spaces are
Banach spaces. In the case $p=2$, the space $A^2_\alpha(\B)$ is a
Hilbert space with the inner product
$$(f,g)_\alpha=\int_{\B}f(z)\overline{g(z)}\, \Id v_\alpha(z).$$
The orthogonal projection of $L^2_\alpha(\B)$ on the space
$A^2_\alpha(\B)$ is given by
$$P_\alpha f(z)=\int_{\B} f(w)K_\alpha(z,w)\, \Id v_\alpha(w),$$
where $$K_\alpha(z,w)=\frac{1}{(1-(z,w))^{n+1+\alpha}} $$ is the
reproducing kernel for $A_\alpha^2(\B) $.

The group $\SU(n,1)$ is defined to be the group of all
$(n+1)\times (n+1)$-matrices $x$ of determinant $1$ for which
$x^*J_{(n,1)}x=J_{(n,1)}$, where
$$J_{(n,1)}=
\begin{pmatrix}
  -I_n & 0\\
  0   &1 
\end{pmatrix}.
$$
We always write $x\in \SU(n,1)$ in the block form 
$$x=
\begin{pmatrix}
           A & b\\
           c^t   &d 
\end{pmatrix},
$$
where $A$ is an $n \times n$ matrix, and $b$, $c$ are vectors in $\C^n$, and $d\in \C$. Simple calculations show that 
$$x^{-1}= 
\begin{pmatrix}
  A^* & -\overline{c}\\
  -\overline{b}^t   &\overline{d}
\end{pmatrix}.
$$ The identity $xx^{-1}=I$ implies
\begin{equation}\label{equ: d2-b2=1}
  |d|^2-|b|^2=1
\end{equation}
Form now on, we write $G=\SU(n,1)$. This group acts
transitively on $\B$ by
$$x \cdot z={(Az+b)}{((c,\overline{z})+d)^{-1}}.$$
If we define the subgroup $K$ of $G$ as
$$K=\left\lbrace \left.
  \begin{pmatrix}
    k&0\\
    0& \overline{\det (k)}
  \end{pmatrix}
  \, \right|\,   k\in U(n)\,\right\rbrace,$$
then the stabilizer of the origin $0 \in \C^n$ is
$K$ and $\B\simeq G/K$. It follows that there is a
one to one correspondence between the $K$-right
invariant functions on $G$ and the functions on
$\B$ via
$$ \ft(x)=f(x\cdot 0).$$
  This correspondence relates the $G$-invariant measure on $\B$, which is
  given by $dv_{-n-1}(z)$, to the measure on the group $G$.
  The compactness of $K$ ensures that we
can normalize the measure on $G$ so that, for any  $K$-right invariant
function $\ft$ on $G$, we have
\begin{equation}\label{eq: integral over G in term of Bn}
  \int_G\ft(x) \, dx= \int _{\B} f(z)\,dv_{-n-1} (z).
\end{equation}
  
Define the
weighted $L^p_\alpha$ spaces on $G$ by
$$L^p_\alpha(G)=\left\lbrace F:G\to \C\, \left|\, 
  \|F\|_{L^p_\alpha(G)}:=\left(C_\alpha \int_G |F(x)|^p\, (1-|x\cdot
    o|^2)^\alpha\, dx\right)^{1/p}<\infty\right. \right\rbrace $$
 
If we denote by $L^p_\alpha(G)^K$ the space of $K$-right invariant
functions in the space $L^p_\alpha(G)$, then it is easy to see that
$L_\alpha^p(\B)$ and $L^p_{\alpha+n+1}(G)^K$ are isometric. That is,
\begin{equation}\label{eq: L(B) and L(G) isometric}
  \|f\|_{L_\alpha^p(\B)}=\|\ft\|_{L^p_{\alpha+n+1}(G)}.
\end{equation}
For $s>n$, the action of $G$ on $\B$ defines an irreducible unitary
projective representation of $G$ on the space $\fH_s=A^2_{s-n-1}$ by
\begin{equation}\label{eq: proj rep of SU(n,1) on H}
  \rho_s(x)f(z)=(-(z,b)+\overline{d})^{-s}f(x^{-1}\cdot z),
\end{equation}
which also defines a representation for the
universal covering group of $G$. 
{}From now on we assume that a cocycle $\sigma$ has been chosen for
the projective representation $\rho_s$. Note, that
since $\rho_s$ is a local representation, the cocycle
can be chosen equal to one on a neighbourhood of $e\times e$
(see also Theorem 7.1 in \cite{Bargmann1954}).

We denote the twisted wavelet transform on $\fH_s$ by
$$W^{\rho_s}_u(\lambda)(x)=(\lambda,\rho_s(x)u)_{\fH_s}.$$
Let $\mathcal{P}_k$ be the space of all homogeneous polynomials of
degree $k$ on $\C^n$. In the following theorem we summarize some
properties of the space of smooth vectors for $\rho_s$ and its
conjugate dual space, which will be the candidate Fr\'echet space $\fS$
for constructing the coorbits of $L^p_{\alpha+n+1-sp/2}(G)$.

\begin{theorem}\label{thm: properties of smooth vectors for proj rep}
  Let $s>n$ and let $(\rho_s,\fH_s)$ be the projective representation
  of $G$ which is defined in \ref{eq: proj rep of SU(n,1) on H}. The
  following are true:
  \begin{enumerate}
  \item Every polynomial is a smooth vector for $\rho_s$.
  \item Every smooth vector for $\rho_s$ is bounded.
  \item Assume $v\in \fH_s$, then $v\in\fHS$ \Iff\, $v=\sum_k v_k$ ,
    $v_k\in \mathcal{P}_k $, and for all $N\in \mathbb{N}$ there
    exists a constant $C_N>0$ such that
    $\|v_k\|_{\fH_s}\leq C_N(1+k)^{-N}$.
  \item A vector $\phi\in \DfHS$ \Iff\, $\phi=\sum_k \phi_k$ ,
    $\phi_k\in \mathcal{P}_k $, and there exist $N\in \mathbb{N}$ and
    $C>0$ such that $\|\phi_k\|_{\fH_s}\leq C(1+k)^{N}$. Moreover, the
    dual pairing is given by
$$\langle\phi, v\rangle_s= \sum_k (\phi_k,v_k)_{\fH_s}.$$
\end{enumerate}
\end{theorem}

\begin{proof}
  The proof is done by noting that $\rho_s$ is a unitary
  representation of the universal covering group of $G$, so the smooth
  vectors are the same for both, where the smooth vectors for
  $\rho_s$, as a representation, satisfy all the above properties as
  proved in \cite{Chebli2004} and \cite[Lemma 2.10]{Christensen2016}.
\end{proof}

\subsection{Bergman Spaces as Coorbits spaces}
In this section we collect several facts from
\cite{Christensen2016} and use them with some modifications to
provide atomic decompositions of Bergman spaces
through projective representations of the group $\SU(n,1)$.

As before, we assume $G= \SU(n,1)$ and $(\fHS,\rho_s)$ is the smooth
projective representation obtained by restricting $(\fH_s,\rho_s)$ to
$\fHS$. In this section we show that Bergman spaces are twisted
convolutive coorbits of weighted $L^p$ spaces, which allows us to
discretize Bergman spaces using the full group $\SU(n,1)$. For this
goal we need the following results which was already proved for the linear
representation in \cite{Christensen2016}. The same proof will work (with minor
differences) for the projective representation case. 
For completeness
we will provide a full proof for each of these results. 
The first lemma corresponds to Lemma 3.15 in \cite{Christensen2016}.
\begin{lemma}\label{Lemma: tW w(v)< Ctw u(u) }
  Assume $u$ and $v$ are smooth vectors for $\rho_s$. There is a
  constant $C$ depending on $u$ and $v$ such that
  $$|W^{\rho_s}_u(v)(x)|\leq C(1-|x\cdot o|^2)^{s/2}\left(1-\log (1-|x\cdot o|^2)\right).$$
Moreover, the constant $C$ can be chosen uniform in $|\alpha|$ for $v=z^\alpha$.
\end{lemma} 

The next result is Proposition 3.16(i) in \cite{Christensen2016} extended to
irrational $s$.

\begin{proposition}\label{twisted wavelet of smooth vector in Lp(G)}
  Let $\alpha>-1$, $1\leq p < \infty$, and $s>n$ be chosen.
  Assume
  that $u$ and $v$ are smooth vectors for $\rho_s$. Then
  $W^{\rho_s}_u(v)\in L^p_t(G)$ for $t+ps/2>n$.
\end{proposition}

We will now verify that any smooth (appropriately normalized) vector $u$
is analyzing. First we will verify the reproducing formula for vectors
in the Hilbert space.
Assume that $v,w\in \fH$.
Since
$\rho_s$ is square integrable and $G$ is unimodular, every 
vector is $\rho_s$-admissible, $i.e.$, $u$ is in the
  domain of the operator $A_\rho$, which is given in Theorem
  \ref{thm:duflomoore}. Note, that $A_\rho$ is a multiple of the identity since $G$ is unimodular.
  By the orthogonality
  relation in Theorem \ref{thm:duflomoore}, we have
$$\int_G ( v,\rho(x)u)_{\fH_s}\,(\rho(x)u,w)_{\fH_s}\, dx=c_\rho^2\|u\|_{\fH_s}^2\,( v,w)_{\fH_s}.$$ 
Letting $w=\rho_s(x)u$ it follows that
 \begin{align*}
  W^{\rho_s}_u(v)\#W^{\rho_s}_u(u)(x)=&\int_G ( v,\rho(y)u)_{\fH_s} \, ( u,\rho(y^{-1}x) u)_{\fH_s}\,\overline{\sigma(y,y^{-1}x)}\, dy\\
   =& \int_G ( v,\rho(y)u)_{\fH_s} \, ( u,\rho(y^{-1})\rho(x) u)_{\fH_s}\,\overline{\sigma(y^{-1},x)}\,\overline{\sigma(y,y^{-1}x)}\, dy \\
  =&\int_G ( v,\rho(y)u)_{\fH_s} \, ( u,\rho(y)^{-1}\rho(x) u)_{\fH_s}\,\sigma(y,y^{-1}) \,\overline{\sigma(y,y^{-1})}\, dy\\
  = &\int_G ( v,\rho(y)u)_{\fH_s} \, ( \rho(y)u,\rho(x) u)_{\fH_s}\, dy \\
  =&c_\rho^2(v,\rho(x)u)_{\fH_s}\, ( u,u)_{\fH_s}  \\
  =& CW^{\rho_s}_u(v)(x) 
 \end{align*}
 for all $v\in \fH$. 
 
   We will now define the twisted convolution in a weak sense
   as described in Remark~\ref{rem:weakconv}.
   Let $0\leq \psi_n\leq 1$ be an increasing sequence of
   compactly supported smooth functions on $G$ that are
   identically $1$ on growing compact sets $C_n$ for which
   $\cup_n C_n = G$. Also assume that
   derivatives of order $N$ of 
   $x\mapsto \ell_x^\sigma \psi_n$ at the origin are uniformly bounded
   in $n$, i.e. for $Y_1,\cdots,Y_N\in \mathfrak{g}$ we have
   $$
   \sup_n \| \ell^\sigma(Y_1)\cdots \ell^\sigma(Y_N) \psi_n\|_\infty < \infty,
   $$
   where $\ell^\sigma(X)f(x) = \frac{d}{dt}\Big|_{t=0} \ell^\sigma_{e^{tX}} f(x)$ are strong derivatives.
   Now define the twisted convolution of two functions $f,g$ by
   \begin{equation}
     \label{eq:weakdefconvsmooth}
     f \# g(x)= \lim_{n\to\infty}
     \int_G \psi_n(y)f(y) \ell_y^\sigma g(x) \, dy 
   \end{equation}
   whenever the limit exists.

 \begin{lemma}
    \label{lemma:reprofordist}
    Let $u$ be a non-zero smooth vector for $\rho_s$ for which
    $\| u\|_{\fH_s} = c_\rho^{-1}$. Then for $\phi \in \fH^{-\infty}_s$ the
    twisted convolution
    $W^{\rho_s}_u(\phi)\# W^{\rho_s}_u(u)$ exists and equals
    $
    W^{\rho_s}_u(\phi).
    $
  \end{lemma}
 
  \begin{proof}
    We get that
    \begin{align*}
    W^{\rho_s}_u(\phi)\# W^{\rho_s}_u(u)(x)
    &= \lim_{n\to\infty} \int
    \psi_n(y) W_u^{\rho_s}(\phi)(y) \overline{\sigma(y,y^{-1}x)}W_u^{\rho_s}(u)(y^{-1}x) \,dy \\
    &= \lim_{n\to\infty} \int
      \psi_n(y) \langle \phi,\rho_s(y)u\rangle \overline{\sigma(y,y^{-1}x) } \langle u, \rho_s(y^{-1}x) u\rangle \,dy \\
      &= \lim_{n\to\infty} \int
      \psi_n(y) \langle \phi,\rho_s(y)u\rangle \langle \rho_s(y)u, \rho_s(x) u\rangle \,dy.
    \end{align*}
    Define the smooth compactly supported function
    $$
    \Psi_n(y) = \psi_n(y)\langle \rho_s(x)u,\rho_s(y)u \rangle,
    $$
    then
    $$
    W^{\rho_s}_u(\phi)\# W^{\rho_s}_u(u)(x)
    = \lim_{n\to\infty} \int \langle \phi,\Psi_n(y)\rho_s(y)u\rangle\, dy.
    $$
    Since $\Psi_n$ is smooth and compactly supported the vector
    $$
    \rho_s(\Psi_n)u := \int \Psi_n(y)\rho_s(y)u, dy
    $$
    is a smooth vector, and
    $$
    W^{\rho_s}_u(\phi)\# W^{\rho_s}_u(u)(x)
    = \lim_{n\to\infty} \langle \phi, \rho_s(\Psi_n)u \rangle.
    $$
    To show that this limit exists and equals
    $\langle \phi,\rho_s(x)u \rangle$,
    we need to verify that
    $\rho_s(\Psi_n)u$ converges to $\rho_s(x)u$ in the
    topology of the smooth vectors $\mathcal{H}_{s}^\infty$.
    Denote $\rho(x)u$ by $v$, then
    we have to show that for
    $$
    \Psi_n(y) = \psi_n(y)\langle v,\rho_s(y)u \rangle,
    $$
    the vectors $\rho_s(\Psi_n)u$ converge to $v$
    in $\mathcal{H}_{s}^\infty$.
    Let us first verify convergence in $\mathcal{H}_{s}$.
    $$
    \| \rho_s(\Psi_n)u - v  \|^2_{\fH_s}
    =
    \| \rho_s(\Psi_n)u\|^2_{\fH_s} + \|  v\|^2_{\fH_s}
    - (\langle \rho_s(\Psi_n)u,  v   \rangle
    - \langle v   , \rho_s(\Psi_n)u \rangle).
    $$
    By the Lebesgue dominated convergence theorem and square integrability
    of the representation, we have that
    $$
    \lim_{n\to\infty}\langle \rho_s(\Psi_n)u,  v   \rangle
    = \lim_{n\to\infty} \int \psi_n(y) |\langle \rho_s(y)u,v \rangle|^2 \, dy
    = \| v\|^2_{\fH_s}.
    $$
    Therefore we just need to check that
    $$
    \lim_{n\to\infty} \| \rho_s(\Psi_n)u\|^2 =  \|  v\|^2_{\fH_s}.
    $$
    Notice that by Fubini we have
    $$
    \| \rho_s(\Psi_n)u\|^2_{\fH_s}
    = \int \int \psi_n(x)\psi_n(y)
    \langle v,\rho_s(y)u\rangle
    \langle \rho_s(y)u ,\rho_s(x)u\rangle
    \langle \rho_s(x)u,v \rangle \, dx\, dy,
    $$
    and we will be able to apply the Lebesgue dominated convergence theorem if
    we can show that the function
    $|\langle v,\rho_s(y)u\rangle
    \langle \rho_s(y)u ,\rho_s(x)u\rangle
    \langle \rho_s(x)u,v \rangle|$ is integrable.
    This is the same as showing that the integral
    $$
    \int \int |W^{\rho_s}_u(v)(y)||W^{\rho_s}_u(u)(y^{-1}x)||W^{\rho_s}_u(v)(x)| \,dx\,dy
    $$
    is finite.
    From Proposition~\ref{twisted wavelet of smooth vector in Lp(G)}
    it is known that $W^{\rho_s}_u(v)\in L^2(G)$ and 
    $W^{\rho_s}_u(u)\in L^p$ for some $1<p<2$, so the Kunze-Stein phenomenon
    \cite{Cowling1978} tells us that the integral is finite.

    We can now repeat the argument with derivatives of the vector
    $\rho_s(\Psi_n)u$ to show it converges to derivatives of
    $v$.
    If $X\in\mathfrak{g}$ then
    $$
    \rho_s(X)\rho_s(\Psi_n)u = \rho_s(X\Psi_n) \rho_s.
    $$
    Since $\Psi_n = \psi_nW_u^{\rho_s}(v)$ we get that
    $$
    \rho_s(X)\rho_s(\Psi_n)u = (\ell^\sigma(X)\psi_n(x))W^{\rho_s}_u(v)
    + \psi_n(x)W_u^{\rho_s}(\rho_s(X)v).
    $$
    As before, it follows from the Lebesgue dominated
    convergence theorem and the assumption that $\ell^\sigma(X)\psi_n$ is
    uniformly bounded in $n$, that
    $\rho_s(X)\rho_s(\Psi_n)u$ converges to $\rho_s(X)v)$
    in $\mathcal{H}_{\rho_s}$. This argument can be repeated to show that
    $\rho_s(\Psi_n)u$ converges to $v$ in $\mathcal{H}_{\rho_s}^\infty$.

 The argument also shows that the twisted convolution
    $W^{\rho_s}_u(\phi)\# W^{\rho_s}_u(u)$ exists and does not depend on
    the particular choice of the sequence of functions $\psi_n$.
  \end{proof}

Now we are ready to show that the twisted coorbits of the spaces
$L^p_{\alpha+n+1-s p/2}(G)$ generated by any nonzero smooth vector
$u\in \fHS$ are well defined nonzero spaces under the assumptions in
the following theorem. 
This result uses techniques found in the proof of
Proposition 3.16 in \cite{Christensen2016}, but needs to be verified in our
situation due to the occurence of the cocycle.

\begin{theorem}\label{twisted coorbit of Bergman space is well defined}
  Let $1\leq p < \infty$, and $s>n$. Assume that
  $-1<\alpha<p(s-n)-1$. For a nonzero smooth vector $u\in\fHS$, the
  coorbit space $\coLp$ is a nonzero well defined Banach space.
\end{theorem}
\begin{proof}
  Let us show that any nonzero smooth vector $u\in \fHS$ satisfies
  Assumption \ref{assumptionrhocoorbit}. First, $u$ is $\rho$-cyclic
  because $\fHS$ is an irreducible projective representation. 
  By the previous lemma it is also analyzing. All that is left to show is
  that the mapping
$$(f,v)\mapsto \int_G
f(x)W^{\rho_s}_u(x^{-1})\overline{\sigma(x,x^{-1})}\, dx$$
is continuous on $L^P_{\alpha+n+1-sp/2}(G)$. By Remark
\ref{Rem:assumptionrhocoorbit}, it is enough to show
that
$$W^{\rho_s}_u(v)\in
(L^p_{\alpha+n+1-sp/2}(G))^*=L^q_{sq/2-(\alpha+n+1)q/p}(G),$$
where $1/p+1/q=1$. This is done by Proposition \ref{twisted wavelet of
  smooth vector in Lp(G)}, because $$sq/2-(\alpha+n+1)q/p+sp/2>n,$$
whenever $-1<\alpha<p(s-n)-1$. Therefore, the space
$\Co_{\rho_s}^u L^p_{\alpha+n+1-sp/2}(G))$ is well defined. Finally,
note that $W^{\rho_s}_u(u)\in L^p_{\alpha+n+1-sp/2}(G)$ again by
Proposition \ref{twisted wavelet of smooth vector in Lp(G)}, hence it
is nonzero Banach space.
\end{proof}

Our goal now is to describe Bergman spaces as twisted coorbits
generated by any nonzero smooth vector $u\in \fHS$. First we describe
Bergman spaces as twisted coorbits by the special $\rho$-analyzing
vector $u=1_{\B}$, then we show that this coorbit is independent of
the choice of $u$. 

\begin{theorem}
  Let $\alpha>-1$, $1\leq p < \infty$, and $u=1_{\B}$. The Bergman
  space $A^p_\alpha(\B)$ is the twisted coorbit space of
  $L^p_{\alpha+n+1-s p/2}(G)$ that corresponds to the projective
  representation $(\fHS,\rho_s)$. $i.e.$,
  $A^p_\alpha(\B)= \Co_{\rho_s}^u L^p_{\alpha+n+1-s p/2}(G)$ for
  $\alpha<p(s-n)-1$.
\end{theorem}
\begin{proof}
  As in \cite[Theorem 3.6 ]{Christensen2016}, the space
  $A^p_\alpha (\B)\subset \DfHS$ for all $p\geq 1$, which is still
  valid in the case of smooth vectors for $\rho_s$. The reason 
  is that the smooth vectors for the projective representation
  $\rho_s$ are the same as the smooth vectors of the 
  representation of the
  universal covering of $G$. 
  First we recall that every element in $ \DfHS$ can be realized 
  as a holomorphic
  function   $f=\sum_k f_k$ in $\DfHS$ (for details see \cite{Christensen2016}). We then have to show that the function $f$ is in  $L^p_\alpha(\B)$ \Iff
  $W^{\rho_s}_u(f)\in L^p_{\alpha+n+1-sp/2}(G)$.
  
  To this end, assume
  $x=\left(\begin{array}{cc}             A&b\\             c^t&d           \end{array}\right)$. 
  By definition of the dual pairing and the fact that the polynomials $f_k$ 
  belong to all Bergman spaces, we have  
  \begin{align*}
    |W^{\rho_s}_u (f)(x)|=&\left|\sum_k  W^{\rho_s}_u(f_k)(x)\right|\\
    =&\left|\sum_k (\overline{d})^{-s}f_k(bd^{-1})\right|\\
    =&|d|^{-s}\left|\sum_kf_k(x\cdot o)\right|\\
    =&|d|^{-s}|f(x\cdot o)|.
  \end{align*}
  As we have seen before, $|d|^{-s}=(1-|x \cdot
  o|^2)^{s/2}$.
  It follows that
  $$|f(x\cdot o)|= (1-|x \cdot o|^2)^{-s/2}|W^{\rho_s}_u
  (f)(x)|$$
  by the isometry in (\ref{eq: L(B) and L(G) isometric}). We
  conclude that $f\in L^p_\alpha(\B)$ \Iff\,
  $W^{\rho_s}_u (f)\in L^p_{\alpha+n+1-sp/2}(G)$.
\end{proof}

To prove our main result in this section, which says that
Bergman spaces are twisted coorbits for weighted $L^p$ spaces
generated by any smooth vector, we need the following
theorem. It will be used in the subsequent section to generate
a Banach frame and atomic decomposition for Bergman spaces.
The following Lemma makes the transition from the weakly defined
twisted convolution to a proper integral which is needed for providing
frames and atoms.
\begin{lemma}
  Let $1\leq p < \infty$, $-1<\alpha<p(s-n)-1$, and let $v$ and
  $u$ be smooth vectors.
  When $f$ is in   $L^p_{\alpha+n+1-sp/2}(G)$
  the twisted convolution
  $f\#W^{\rho_s}_u(v)$ is a proper integral, i.e.
  $$
  f\#W^{\rho_s}_u(v)(x) = \int_G f(y) \ell_y^\sigma W^{\rho_s}_u(v) (x) \, dy.
  $$
\end{lemma}
\begin{proof}
  Just an application of Lebesgue's dominated convergence theorem
  and the fact that $W^{\rho_s}_u(v)$ is in the dual of
  $L^p_{\alpha+n+1-sp/2}(G)$ for the specified parameters.
\end{proof}

\begin{theorem}\label{Thm:convo is cont. on B, proj case}
  Let $1\leq p < \infty$, $-1<\alpha<p(s-n)-1$, and let $v$ and
  $u$ be smooth vectors. The convolution operator
  $f\mapsto f*|W^{\rho_s}_u(v)|$ is continuous on
  $L^p_{\alpha+n+1-sp/2}(G)$. In particular,
  $f\mapsto f\#W^{\rho_s}_u(v)$ is continuous on
  $L^p_{\alpha+n+1-sp/2}(G)$.
\end{theorem}
In \cite{Christensen2016} this result was proved for a solvable
and simply connected subgroup of $G$ in Corollary 3.10. The proof below
is similar, but we include it for completeness.
\begin{proof}
  Let $F\in L^p_{\alpha+n+1-sp/2}(G)$ and
  define $$\ft(x):=\int_K F(xk)\Id k.$$ Then $\ft$ is $K$-right
  invariant function on $G$. Therefore, there is a
  corresponding $f\in L^p_{\alpha-sp/2}(\B)$. Now, for
  $\epsilon>0$ small enough such that
  $-(s-\epsilon)p/2<\alpha-sp/2+1<p\left((s-\epsilon)/2-n\right)-1
  $
  whenever $-sp/2<\alpha-sp/2+1<p\left(s/2-n\right)-1 $, we
  have
  \begin{align*}
    |F|*|W^{\rho_s}_u(v)|(x)=& \int_G |F(y)|\,|W^{\rho_s}_u(v)(y^{-1}x)|\, dy\\ 
    \leq &C \int_G |F(y)|(1-|y^{-1}x\cdot o|^2)^{s/2}|1-\log(1-|y^{-1}x\cdot o|^2)|\, dy\\
    \leq& C  \int_G |F(y)|(1-|y^{-1}x\cdot o|^2)^{(s-\epsilon)/2}\, dy\\ 
    = &C  \int_{G/K} |\ft(y)|(1-|y^{-1}x\cdot o|^2)^{(s-\epsilon)/2}\, dy.
  \end{align*}
  If we assume that
  $x=\left(\begin{array}{cc} A_x & b_x \\ c_x^t&d_x           \end{array} \right)$ ,  
  $y=\left(\begin{array}{cc}              A_y& b_y\\ c_y^t&d_y           \end{array} \right)$, 
  $w=x\cdot o = b_xd_x^{-1}$, and $z=y\cdot o= b_yd_y^{-1}$,
  then $$d_{y^{-1}x}=\overline{d}_yd_x(1-( w,z) )$$ and $$|d_x|^{-(s-\epsilon)}=(1-|x\cdot o|^2)^{(s-\epsilon)/2}.$$ Therefore,
  \begin{align*}
    (1-|y^{-1}x\cdot o|^2)^{(s-\epsilon)/2} 
    =&|d_{y^{-1}x}|^{-(s-\epsilon)}\\
    =&(1-|x\cdot o|^2)^{(s-\epsilon)/2}\, (1-|y\cdot o|^2)^{(s-\epsilon)/2}
       (1-(x\cdot o,y\cdot o))^{-(s-\epsilon)}.
  \end{align*}
  Thus,
  \begin{align*}
    |F|*|W^{\rho_s}_u(v)|(x)=& C  \int_{G/K} |\ft(y)|\,\frac{(1-|x\cdot o|^2)^{(s-\epsilon)/2}\, 
                               (1-|y\cdot o|^2)^{(s-\epsilon)/2}}{|1-( x\cdot o,y\cdot o ) |^{(s-\epsilon)}} \, dy\\ 
    =&C  (1-|w|^2)^{(s-\epsilon)/2}\int_{\B} |f(z)|\, \frac{(1-|z|^2)^{(s-\epsilon)/2-n-1}}{|1-( w,z ) |^{(s-\epsilon)}}\, dz.
  \end{align*}
  According to
  \cite[Theorem
  2.10]{Zhu2005}, the
  operator $S$ which is
  given
  by
  $$Sf(z)=(1-|w|^2)^{(s-\epsilon)/2}\int_{\B}
  |f(z)|\,
  \frac{(1-|z|^2)^{(s-\epsilon)/2-n-1}}{|1-(
    w,z )
    |^{(s-\epsilon)}}\,
  dz $$
  is continuous on
  $L^p_{\alpha-sp/2}(\B)$
  whenever
  $$-(s-\epsilon)p/2<\alpha-sp/2+1<p\left((s-\epsilon)/2-n\right)-1,
  $$
  which is equivalent
  to
  $-1<\alpha<p(s-n)-1$. Since
  $$\|f\|_{L^p_{\alpha-sp/2}(\B)}=\|\ft\|_{L^p_{\alpha+n+1-sp/2}(G/K)}=\|F\|_{L^p_{\alpha+n+1-sp/2}(G)},$$
  the operator
  $F\mapsto
  F*|W^{\rho_s}_u(v)|$
  is continuous on
  $L^p_{\alpha+n+1-sp/2}(G)$. The
  second part is clear
  from the relation
  $
  |F\#W^{\rho_s}_u(v)(x)|\leq
  |F|*|W^{\rho_s}_u(v)(x)|$.
\end{proof}

We conclude our section
with the following main
result 
which extends \cite[Proposition 3.16(v)]{Christensen2016}
to the projective representation for irrational $s$.

\begin{theorem}
  Let $1\leq p < \infty$ and $-1<\alpha<p(s-n)-1$, and let $v\in \fHS$
  be a nonzero smooth vector. The Bergman space $A^p_\alpha(\B)$ is
  the twisted coorbit space of $L^p_{\alpha+n+1-s p/2}(G)$ via the
  projective representation $(\fHS,\rho_s)$. That is,
  $A^p_\alpha(\B)= \Co_{\rho_s}^v L^p_{\alpha+n+1-s p/2}(G)$ for
  $\alpha<p(s-n)-1$.
\end{theorem}
\begin{proof}
  Assume $u=1_{\B}$. By Theorem \ref{Thm:convo is cont. on B, proj
    case}, we have
  $A^p_\alpha(\B)= \Co_{\rho}^u L^p_{\alpha+n+1-s p/2}(G)$. We will
  show that the twisted coorbit
  $ \Co_{\rho_s}^v L^p_{\alpha+n+1-s p/2}(G)$ does not depend on the
  analyzing vector $v$, by applying Theorem \ref{thm:u-indep}. First,
  according to Theorem \ref{Thm:convo is cont. on B, proj case}, the
  operators $ f\mapsto f\#W^{\rho_s}_u(v)$ and
  $ f\mapsto f\#W^{\rho_s}_v(u)$ are continuous on
  $L^p_{\alpha+n+1-s p/2}(G)$.
  Next, we show that
  $W^{\rho_s}_u(\phi)\#W^{\rho_s}_v(u)=CW^{\rho_s}_v(\phi)$ for all
  $\phi\in \DfHS$. For $f\in\fHS$, we can use the orthogonality
  relation in Theorem \ref{thm:duflomoore} to get
  $W^{\rho_s}_u(f)\#W^{\rho_s}_v(u)= C W^{\rho_s}_v(f)$. To extend
  this relation to the dual of the smooth vectors, it is enough to
  show that
  $$\phi\mapsto \int_G \left\langle \phi ,\rho(x)u\right\rangle\,
  \left\langle \rho(x)v,u\right\rangle\,dx$$
  is weakly continuous. Same argument, as in the proof of Theorem
  \ref{twisted coorbit of Bergman space is well defined}, can be made
  to show our claim. Therefore, the twisted coorbit spaces
  $\Co_{\rho_s}^v L^p_{\alpha+n+1-s p/2}(G)$ are all equal to the space
  $A^p_\alpha(\B)$.
\end{proof}

\subsection{New atomic decompositions and frames 
  for Bergman Spaces}
\noindent
In this section we generate a wavelet frame and an atomic
decomposition of Bergman spaces depending on the coorbit theory, where
this discretization would work for all projective representations with
$s>n$, including the non-integrable cases. Also, we have more freedom
in choosing the wavelet $u$. That is we show that any nonzero smooth
vector can be used to generate a Banach frame and an atomic
decomposition for Bergman spaces.
This result removes the restriction of $s>n$ being rational in 
Theorem 3.17 from \cite{Christensen2016}. Also, it completely avoids
the use of covering groups of $G$ that was present in that paper, essentially
generating atoms from points in $G$ rather than points on a cover.

\begin{theorem}
  Assume that $1\leq p<\infty$, $s>n$, and $-1<\alpha<p(s-n)-1$. For a
  nonzero smooth vector $u$ for $\rho_s$, we can choose $\epsilon$
  small enough such that for every $U_\epsilon$-well spread set
  $\{x_i\}_{i\in I}$ in $G$ the following hold.
  \begin{enumerate}
  \item(Twisted wavelet frame) The family $\{\rho_s(x_i)u:\, i\in I\}$
    is a Banach frame for $A^p_\alpha(\B)$ with respect to the
    sequence space $\ell^p_{\alpha+n+1-ps/2}(I)$. That is, there exist
    constants $A,B>0$ such that for all $f\in A^p_\alpha(\B)$ we have
$$A\|f\|_{A^p_\alpha(\B)}\leq \|\{\left\langle f,\rho_s(x_i)u\right\rangle\}\|_{\ell^p_{\alpha+n+1-s p/2}(I)}\leq B \|f\|_{A^p_\alpha(\B)},$$ and $f$ can be reconstructed by 
$$
f = (W^{\rho_s}_u)^{-1} S^{-1}_1 \Big(\sum_i W^{\rho_s}_u(f)(x_i) \psi
_i \# W^{\rho_s}_u(u)\Big)$$
where $\{\psi_i\}$ is any $U_\epsilon$-BUPU with
$\supp \psi_i\subset x_iU_\epsilon$.
\item(Atomic decomposition) There exists a family of functionals
  $\{\gamma_i\}_{i\in I}$ on $A^p_\alpha(\B) $ such that the family
  $\{\gamma_i, \rho_s(x_i)u\}$ forms an atomic decomposition for
  $A^p_\alpha(\B)$ with respect to the sequence space
  $\ell^p_{\alpha+n+1-ps/2}(I)$, so that any $f\in A^p_\alpha(\B)$ can
  be reconstructed by
$$ 
f = \sum_i \gamma_i(f)\,\rho_s(x_i)u .$$
\end{enumerate}
\end{theorem}
\begin{proof}
  We show that the assumptions of Theorem \ref{Thm:atomicdecomp} are
  satisfied. Under the conditions on $p$ and $s$, the twisted coorbit
  of $L^p_{\alpha+n+1-ps/2}(G)$ is well defined and $u$ satisfies
  Assumption \ref{assumptionrhocoorbit} as we have seen in Theorem
  \ref{twisted coorbit of Bergman space is well defined}, and it is
  equal to $A^p_\alpha(\B)$. Since $u$ is smooth vector for $\fH_s$,
  and $\fHS$ is continuously embedded in its dual $\DfHS$, the vector
  $u$ is $\rho$ - and $\rho^*$-weakly differentiable. According to
  Theorem \ref{Thm:convo is cont. on B, proj case}, the
  mappings
  $$f\mapsto f*|W^{\rho_s}_{\rho(E_\alpha)u}(u)| \quad \text{and}
  \quad f\mapsto f*|W^{\rho_s}_u(\rho^*(E_\alpha) u)|$$
  are continuous on $L^p_{\alpha+n+1-ps/2}(G)$. Therefore, we can
  choose $\epsilon$ small enough so that the family
  $\{\rho_s(x_i)u\}$ forms a frame and an atomic decomposition for
  $A^p_\alpha(\B)$ with reconstruction operators that are given in
  Theorem \ref{Thm:atomicdecomp}.
\end{proof}

\appendix

\section{Decompositions of reproducing kernel spaces for twisted convolution}
\label{appendixA}
\noindent
{}From now on we let $\phi(x)=W^\rho_u(u)$ for some fixed $u\in\fS.$
Then $B_u^\sigma = \{ f\in B \mid f=f\#\phi \}$.  Given a compact
neighbourhood $U$ of the identity in $G$, a $U$-dense and well-spread
sequence $\{ x_i\}\subseteq G$ and a 
$U$-BUPU $\{\psi_i \}$, we formally define
the operators
\begin{align*}
  T_1 f &= \sum_i f(x_i) \sigma(x,x^{-1}x_i)\psi_i\#\phi \\
  T_2 f &= \sum_i \lambda_i(f) \ell_{x_i}^\sigma \phi \\
  T_3 f &= \sum_i c_i f(x_i) \ell_{x_i}^\sigma \phi
\end{align*}
where
$\lambda_i(f) = \int f(y)\psi_i(y)\overline{\sigma(y,y^{-1}x_i)}\,dy$
and $c_i = \int \psi_i(y)\,dy$. 
The following results will establish when these operators
are well defined on $B_u^\sigma$.

Define the local oscillations
$$
\osc_{U}^{r^\sigma} f(x) = \sup_{y\in U} | r^\sigma_y f(x)-f(x)|
\quad\text{and} \quad \osc_{U}^{\ell^\sigma} f(x) = \sup_{y\in U} |
\ell^\sigma_y f(x)-f(x)|.
$$

\begin{proposition}
  If $f\in B_u^\sigma$ then
  \begin{align*}
    |T_1f(x)-f(x)| &\leq |f|*\osc^{r^\sigma}_{U^{-1}} \phi(x) \\
    |T_2f(x)-f(x)| &\leq |f|*\osc^{\ell^\sigma}_{U^{-1}} \phi(x) \\
    |T_3f(x)-f(x)| &\leq |f|*\osc^{r^\sigma}_{U^{-1}} \phi*(|\phi|
                     +\osc^{\ell^\sigma}_{U^{-1}} \phi)(x) 
                     + |f|*\osc_{U^{-1}}^{\ell^\sigma}\phi(x).
  \end{align*}
\end{proposition}

\begin{proof}
  We see that
  $$
  |\sum_i f(x_i) \sigma(x,x^{-1}x_i)\psi_i(x) \\-f(x)| \leq \sum_i
  |f(x_i)\sigma(x,x^{-1}x_i)-f(x)|\psi_i(x)
  $$
  and for $x\in x_iU$ we get that $x_i \in xU^{-1}$, so
  $$
  |f(x_i)\sigma(x,x^{-1}x_i)-f(x)| \leq \sup_{y\in U^{-1}} |
  f(xy)\sigma(x,y) - f(x)| = \osc_{U^{-1}}^{r^\sigma} f(x).
  $$
  Next, if $f=f\#\phi$ we get
  \begin{equation*}
    |r^\sigma_y f(x) - f(x)|
    \leq \int |f(z)||\phi(z^{-1}xy)\overline{\sigma(z,z^{-1}xy)}\sigma(x,y)
    - \phi(z^{-1}x)\overline{\sigma(z,z^{-1})}|\,dz.
  \end{equation*}
  Since
  $\sigma(z,z^{-1}xy) =
  \sigma(z,z^{-1}x)\sigma(x,y)\overline{\sigma(z^{-1}x,y)}$
  we get that the integral above reduces to
  $$
  \int |f(z)||r^\sigma_y \phi(z^{-1}x) - \phi(z^{-1}x)|\,dz.
  $$
  Taking supremum we get the desired result.

  We have that
  \begin{align*}
    f(x)-T_2f(x)
    &=\int f(y) \ell^\sigma_y\phi(x)\,dy
      - \int \sum_i f(y)\psi_i(y)\overline{\sigma(y,y^{-1}x_i)}\,dy 
      \ell^\sigma_{x_i} \phi(x) \\
    &=
      \int \sum_i f(y)\psi_i(y)[\ell^\sigma_y\phi(x)-
      \overline{\sigma(y,y^{-1}x_i)}
      \ell^\sigma_{x_i} \phi(x)]\,dy  \\
    &=
      \int \sum_i f(y)\psi_i(y)\ell^\sigma_y
      [\phi(x)-\ell^\sigma_{y^{-1}x_i} \phi(x)]\,dy.
  \end{align*}
  When $y\in x_iU$ for a compact neighbourhood $U$ of the identity,
  then $ y^{-1}x_i \in U^{-1}$.  Thus we get
  \begin{align*}
    |f(x)-T_2f(x)|
    &\leq
      \int \sum_i |f(y)| \psi_i(y)|\ell^\sigma_y
      \phi(x)-\ell^\sigma_{y^{-1}x_i} \phi(x)|\,dy \\
    &\leq
      \int \sum_i |f(y)| \psi_i(y)|\ell^\sigma_y
      \osc_{U^{-1}}^{\ell^\sigma} \phi |\,dy \\
    &\leq
      \int |f(y)| \ell_y \osc_{U^{-1}}^{\ell^\sigma} \phi \,dy.
  \end{align*}
  This shows the claim.
  
  \begin{equation*}
    |T_3f(x) - f(x)|
    \leq \int \sum_i \psi_i(y) |f(x_i)\ell^\sigma_{x_i}\phi(x) 
    - f(y)\ell^\sigma_y\phi(x)|\,dy
  \end{equation*}
  Let us rewrite part of the integrand when $y\in x_iU$
  \begin{align*}
    |&f(x_i)\ell^\sigma_{x_i}\phi(x) - f(y)\ell^\sigma_y\phi(x)| \\
     &= |f(x_i)\sigma(y,y^{-1}x_i)
       \ell^\sigma_y\ell^\sigma_{y^{-1}x_i}\phi(x) - f(y)\ell^\sigma_y\phi(x)| \\
     &\leq |[f(x_i)\sigma(y,y^{-1}x_i)-f(y)]
       \ell^\sigma_y\ell^\sigma_{y^{-1}x_i}\phi(x)|
       + |f(y)||\ell^\sigma_y\ell^\sigma_{y^{-1}x_i}\phi(x)\ 
       - \ell^\sigma_y\phi(x)| \\
     &= |[r^\sigma_{y^{-1}x_i}f(y)-f(y)]
       \ell^\sigma_y\ell^\sigma_{y^{-1}x_i}\phi(x)|
       + |f(y)||\ell^\sigma_y\ell^\sigma_{y^{-1}x_i}\phi(x)\ 
       - \ell^\sigma_y\phi(x)| \\
     &\leq \osc^{r^\sigma}_{U^{-1}}f(y)
       [|\ell^\sigma_y\ell^\sigma_{y^{-1}x_i}\phi(x) 
       - \ell^\sigma_y \phi(x)|+|\ell^\sigma_y\phi(x)|]
       + |f(y)||\ell^\sigma_y \osc^{\ell^\sigma}_{U^{-1}} \phi(x)| \\
     &\leq \osc^{r^\sigma}_{U^{-1}}f(y)
       [|\ell^\sigma_y\osc^{\ell^\sigma_{U^{-1}}} \phi(x)|
       +|\ell^\sigma_y\phi(x)|]
       + |f(y)||\ell^\sigma_y \osc^{\ell^\sigma}_{U^{-1}} \phi(x)|.
  \end{align*}
  As before the oscillation of $f$ can be transferred onto the kernel,
  and the final result is obtained.
\end{proof}

From this we obtain
\begin{corollary}
  \label{cor:convosc}
  Let $B$ be a solid BF-space and assume that $f\mapsto f*|\phi|$,
  $f\mapsto f*\osc^{\ell^\sigma}_{U^{-1}}\phi$ and
  $f\mapsto f*\osc^{r^\sigma}_{U^{-1}}\phi$ are bounded on $B$, then
  $T_1,T_2,T_3$ are well-defined bounded operators on $B^\#_u$.
  
  Moreover, if there are constants $C_U$ for which
  $$
  \| f*\osc^{\ell^\sigma}_{U^{-1}}\phi\| \leq C_U \| f\| \text{ and }
  \| f*\osc^{r^\sigma}_{U^{-1}}\phi\| \leq C_U \| f\|
  $$
  and $\lim_{U\to \{ e\}} C_U =0$, then there is a $U$ small enough as
  well as $U$-dense $\{ x_i\}$ such that the operators are invertible
  on $B^\#_u$.
\end{corollary}

We will now use the special form of $\phi$ to find oscillation
estimates via derivatives. We have defined
$\phi(x) = \langle u, \rho(x)u\rangle$, and from this and
Remark~\ref{rem:leftandright} we get
\begin{align*}
  \osc_U^{r^\sigma} \phi(x) 
  &= \sup_{y\in U} |\langle \rho^*(x^{-1})u,\rho(y)u-u\rangle|, \\
  \intertext{and}
  \osc_U^{\ell^\sigma} \phi(x) 
  &= \sup_{y\in U} |\langle \rho^*(y)u-u, \rho(x)u\rangle|.
\end{align*}
In light of this is seems possible to evaluate the oscillation by a
certain level of smoothness of the vector $u$, and this is exactly the
approach we will take.  We let $v\in \mathcal{S}$ and
$\lambda\in \mathcal{S}^*$ be arbitrary elements and define the
functions $H(y) = \langle \lambda , \rho(y)u\rangle$ and
$K(y) = \langle \rho^*(y)u,v\rangle.$ We will now investigate the
local oscillations of $H$ and $K$ in terms of derivatives, but first
we need to introduce some notation.

If $f$ is a function on $G$ and $X$ is in $\fg$, then define
$$
Xf(y) = \frac{d}{d t}\Big|_{t=0} f(y\exp(tX)).
$$
We now fix a basis $X_1,\dots,X_n$ for the Lie algebra $\fg$, and for
a multi-index $\alpha$ we define
$$
X^\alpha f = X_1^{\alpha_1}\cdots X_n^{\alpha_n} f.
$$
We will investigate oscillations of $H$ and $K$ on the specific
neighbourhood
$$
U_\epsilon = \{ \exp(t_1X_1)\cdots \exp(t_nX_n)\mid -\epsilon \leq
t_k\leq \epsilon\}.
$$
Remember that we choose the cocycle $\sigma$ and $\epsilon>0$
such that $\sigma$ is $C^\infty$ on a neighbourhood containing
$U_\epsilon\times U_\epsilon$.
According to Lemma 2.5 in \cite{Christensen2012} there is a constant $C_\epsilon$
such that
\begin{align*}
  \sup_{y\in U_\epsilon} |H(y)-H(e)|
  &\leq C_\epsilon \sum_{\stackrel{1\leq |\alpha|\leq n}{|\delta|=|\alpha|}}
    \int_{[-\epsilon,\epsilon]^{|\delta|}} |X^\alpha H (\tau_\delta(t_1,\dots,t_n))| 
    (dt_1)^{\delta_1} \dots (dt_n)^{\delta_n}, \\
  \intertext{and}
  \sup_{y\in U_\epsilon} |K(y)-K(e)|
  &\leq C_\epsilon \sum_{\stackrel{1\leq |\alpha|\leq n}{|\delta|=|\alpha|}}
    \int_{[-\epsilon,\epsilon]^{|\delta|}} |X^\alpha K (\tau_\delta(t_1,\dots,t_n))| 
    (dt_1)^{\delta_1} \dots (dt_n)^{\delta_n},
\end{align*}
where
$\tau_\delta(t_1,\dots,t_n) =
\exp(\delta_1t_1X_1)\cdots\exp(\delta_nt_nX_n)$
for a multi-index $\delta$.  Due to the special form of $H$
$$
XH(y) = \frac{d}{dt}\Big|_{t=0} \langle \lambda, \rho(y\exp(tX))
u\rangle = \frac{d}{dt}\Big|_{t=0} \langle
\lambda,\rho(y)\rho(\exp(tX))u\rangle \overline{\sigma(y,\exp(tX))}.
$$
Therefore $X^\delta H(y)$ can be expressed as a sum
$$
\sum_{|\gamma|\leq |\delta|} \langle \lambda, \rho(y)
\rho(X_n)^{\gamma_n}\cdots\rho(X_1)^{\gamma_1}u \rangle g_\gamma(y),
$$
where $g_\gamma$ is an appropriate derivative of the cocycle $\sigma$
of order $|\gamma|$. Notice, that the $g_\gamma$'s do not depend on
the vectors $v$ and $\lambda$ used to define $H$ and $K$.  If $y$ is
in the compact set $U_\epsilon$ the functions $g_\gamma$ are uniformly
bounded, and therefore there is a constant $D_\epsilon$ such that
\begin{align*}
\sup_{y\in U_\epsilon} &|H(y)-H(e)| \\
&\leq D_\epsilon \sum_{\stackrel{1\leq
    |\alpha|\leq n}{|\delta|=|\alpha|}}
\int_{[-\epsilon,\epsilon]^{|\delta|}} |\langle
\lambda,\rho(\tau_\delta(\mathbf{t})) \rho(X_n)^{\delta_n}\cdots
\rho(X_1)^{\delta_1} u)| (dt_1)^{\delta_1} \dots (dt_n)^{\delta_n},
\end{align*}
when we write $\mathbf{t}=(t_1,t_2,\dots,t_n)$.
From this we get that
\begin{align*}
\osc^{r^\sigma}_U &\phi(x) \\ 
&\leq D_\epsilon \sum_{\stackrel{1\leq |\alpha|\leq n}{|\delta|=|\alpha|}}
 \int_{[-\epsilon,\epsilon]^{|\delta|}}
|\langle u,\rho(x\tau_\delta(\mathbf{t}))
\rho(X_n)^{\delta_n}\cdots \rho(X_1)^{\delta_1} u)| (dt_1)^{\delta_1}
\dots (dt_n)^{\delta_n}.
\end{align*}
Treating $K$ the same way we get
\begin{align*}
&\osc^{\ell^\sigma}_U \phi(x) \\
&\quad\leq D_\epsilon \sum_{\stackrel{1\leq |\alpha|\leq  n}{|\delta|=|\alpha|}}
  \int_{[-\epsilon,\epsilon]^{|\delta|}}
|\langle \rho^*(X_n)^{\delta_n} \cdots \rho^*(X_1)^{\delta_1} u
,\rho(\tau_\delta(\mathbf{t})^{-1}x)u )| (dt_1)^{\delta_1} \dots
(dt_n)^{\delta_n}.
\end{align*}

\begin{lemma}
  Assume that $B$ is a solid BF-space on which left and right
  translations are continuous.  If
  $f\mapsto f*|\langle u,\rho(\cdot)\rho(X_n)^{\delta_n}\cdots
  \rho(X_1)^{\delta_1} u\rangle|$
  and
  $f\mapsto f*|\langle \rho^*(X_n)^{\delta_n}\cdots
  \rho^*(X_1)^{\delta_1} u,\rho(\cdot)u\rangle|$
  are bounded on $B$ for all $|\delta|\leq \dim(G)$, then
  $$
  \| f*\osc_U^{\ell^\sigma} \phi\|_B \leq C_\epsilon \| f\|_B
  $$
  and
  $$
  \| f*\osc_U^{r^\sigma} \phi\|_B \leq C_\epsilon \| f\|_B.
  $$
  Moreover, $\lim_{\epsilon\to 0} C_\epsilon =0$.
\end{lemma}
\begin{proof}
Write $\mathbf{t}=(t_1,\ldots ,t_n)$.  Notice that
 \begin{align*}
  |f*&\osc_U^{\ell^\sigma} \phi (x) | \\
  &\leq D_\epsilon \sum_{\stackrel{1\leq
    |\alpha|\leq n}{|\delta|=|\alpha|} }
  \int_{[-\epsilon,\epsilon]^{|\delta|}}
  |\ell_{\tau_\delta(\mathbf{t})} f|*|W_{\rho^*(X_n)^{\delta_n}
    \cdots \rho^*(X_1)^{\delta_1} u}(u)(x) )| (dt_1)^{\delta_1} \dots
  (dt_n)^{\delta_n}.
 \end{align*}
  Since left translation is continuous on $B$ the right hand side
  defines a function in $B$ by Theorem 3.29 in \cite{Rudin1991}, and by
  solidity $f*\osc_U^{\ell^\sigma} \phi (x)$ is also in $B$.
  Moreover,
 $\| f*\osc_U^{\ell^\sigma} \phi \| \leq C_\epsilon \| f\|$, where
  $C_\epsilon$ is equal to $D_\epsilon$ multiplied by a polynomial in
  $\epsilon$ with no constant term. Since $D_\epsilon$ is uniform in
  $\epsilon$ we see that $\lim_{\epsilon\to 0} C_\epsilon =0$.
  
  The proof for convolution with right oscillations follows in a
  similar manner.
\end{proof}

\bibliographystyle{abbrv}
\bibliography{projrep}

\end{document}